\documentclass[11pt, letterpaper]{amsart}
\usepackage{amsthm}
\usepackage{amsmath,amssymb,amsfonts,nicefrac,verbatim}
\usepackage[pdfencoding=auto, psdextra]{hyperref}
\usepackage{color}

\usepackage{bm}
\usepackage{bbm}
\usepackage[toc]{appendix}
\usepackage{enumitem}

\newtheorem{thm}{Theorem}[section]

\newtheorem{lem}[thm]{Lemma}
\newtheorem{cor}[thm]{Corollary}

\newtheorem{proposition}[thm]{Proposition}
\newtheorem{definition}[thm]{Definition}

\newtheorem{problem}[thm]{Problem}

\makeatletter
\newtheorem*{rep@theorem}{\rep@title}
\newcommand{\newreptheorem}[2]{%
\newenvironment{rep#1}[1]{%
\def\rep@title{\bf #2 \ref*{##1} \text{(Restated)} }%
\begin{rep@theorem} }%
{\end{rep@theorem} } }

\newtheorem*{rep@claim}{\rep@title}
\newcommand{\newrepclaim}[2]{%
\newenvironment{rep#1}[1]{%
\def\rep@title{\bf #2 \ref*{##1} \text{(Restated)} }%
\begin{rep@claim} }%
{\end{rep@claim} } }

\newtheorem*{rep@lemma}{\rep@title}
\newcommand{\newreplemma}[2]{%
\newenvironment{rep#1}[1]{%
\def\rep@title{\bf #2 \ref*{##1} \text{(Restated)} }%
\begin{rep@lemma} }%
{\end{rep@lemma} } }

\makeatother
\newreptheorem{theorem}{Theorem}
\newrepclaim{claim}{Claim}
\newreplemma{lemma}{Lemma}

\newcommand{\remove}[1]{}

\newcommand\eps{\varepsilon}

\renewcommand\ge{\geqslant}
\renewcommand\le{\leqslant}

\newcommand{\htt}{\operatorname{ht}}

\renewcommand{\epsilon}{\eps}

\newcommand{\NN}{\mathbb{N}}

\newcommand{\rom}[1]{\uppercase\expandafter{\romannumeral #1\relax}}

\date{}

\begin{document}
\title[Hypercontractivity and Character Bounds]{Bounds for Characters of the Symmetric Group: A Hypercontractive Approach}

\author{Noam Lifshitz}
\thanks{Affiliation of the first author: Einstein institute of Mathematics, Hebrew University. \texttt{noamlifshitz@gmail.com}. Supported by the Israel Science Foundation (grant no.~1980/22).}
\author{Avichai Marmor}
\thanks{Affiliation of the second author: Department of Mathematics, Bar-Ilan University. \texttt{avichai@elmar.co.il}. Partially supported by the Israel Science Foundation, Grant No.\ 1970/18, and by the European Research Council under the ERC starting grant agreement No.\ 757731 (LightCrypt).}

\maketitle
\begin{abstract}
   Finding upper bounds for character ratios is a fundamental problem in asymptotic group theory. Previous bounds in the symmetric group have led to remarkable applications in unexpected domains.
    The existing approaches predominantly relied on algebraic methods, whereas our approach combines analytic and algebraic tools. Specifically, we make use of a tool called `hypercontractivity for global functions' from the theory of Boolean functions. By establishing sharp upper bounds on the $L^p$-norms of characters of the symmetric group, we improve existing results on character ratios from the work of Larsen and Shalev [Larsen, M., Shalev, A. Characters of symmetric groups: sharp bounds and applications. Invent. math. 174, 645–687 (2008)]. 
   We use our norm bounds to bound Fourier coefficients of class functions, product mixing of normal sets, mixing time of normal Cayley graphs, and Kronecker coefficients. Our approach bypasses the need for the $S_n$-specific Murnaghan--Nakayama rule. Instead we leverage more flexible representation theoretic tools, such as Young's branching rule, which potentially extend the applicability of our method to groups beyond $S_n$.
\end{abstract}

\section{Introduction}

In this paper, we introduce tools from analysis of Boolean functions and apply them to the study of character ratios in symmetric groups. Since the seminal work of Diaconis and Shahshahani \cite{diaconis1981generating} from 1981, the problem of estimating the character ratios $\frac{\chi(\sigma)}{\chi(1)}$ of a finite group became one of the most prominent problems in asymptotic group theory. For instance, according to Gurevich and Howe \cite{gurevich2021harmonic}, it is the core open problem of harmonic analysis over finite groups. The works of Roichman \cite{roichman1996upper}, M\"{u}ller--Schlage-Puchta~\cite{muller2007character}, and Larsen--Shalev~\cite{larsen2008characters} each presented new bounds for character ratios, and found applications in a wide variety of unexpected areas. These areas include the theory of Fuchsian groups, Cayley graphs with random generators, Waring type problems in finite simple groups, and mixing times and covering numbers corresponding to sets closed under conjugation.

Our main result provides a new upper bound on the character ratios $\left|\frac{\chi(\sigma)}{\chi(1)}\right|$ in $S_n$ in terms of the number of cycles of $\sigma$ (see Theorem \ref{thm: few cycles} below). We prove essentially the best possible upper bounds on the character ratios in terms of a certain parameter of $\chi$ that we call \emph{level} and the number of cycles of $\sigma$. Our notion of level deviates slightly from the one introduced by Kleshchev, Larsen and Tiep~\cite{kleshchev2022level}, which we refer to as \emph{strict level}. It turns out that our level parameter is better at capturing the true nature of the character ratios compared to the dimension of the corresponding irreducible representations.

Our technique differs substantially from previous approaches which heavily depended on explicit formulas for the characters, such as the Murnaghan--Nakayama rule.
In this paper, we demonstrate the power of a tool from the theory of Boolean functions known as `hypercontractivity for global functions'~\cite{keevash2021global,khot2018pseudorandom} in this context. A variant of the hypercontractivity theorem for global functions over the symmetric group was first presented by Filmus, Kindler, Lifshitz and Minzer~\cite{filmus2020hypercontractivity}, and reached its full power in a recent work of Keevash and Lifshitz~\cite{keevash2023sharp}.
Hypercontractivity for global functions is the main tool that enables our improved bounds, and it also has an extra benefit: our method is much more versatile than the $S_n$-specific Murnaghan--Nakayama rule. (We only use the Murnaghan--Nakayama rule in verifying a few tightness examples, such as Theorem~\ref{thm: lower bound few cycles}.) In order to obtain our character bounds the only tool from the representation theory of $S_n$ that we apply is
Young's branching rule, which allows us to decompose restricted representation into irreducible constituents. While explicit formulas for the characters are rare in other groups, the so called `branching problem' is better understood.

Before delving into our results about character ratios in Section \ref{subsec:intro character ratios}, we first explore their applications to mixing times of Cayley graphs.
\subsection{A sample application for mixing times of Cayley graphs}
\label{subsec:intro sample application}

In 1984 Diaconis and Shahshahani \cite{diaconis1981generating} determined the mixing time of the lazy random walk of the normal Cayley graph $\mathrm{Cay}\left(S_n, (ij)^{S_n}\right)$ corresponding to the transpositions by computing their character ratios (a Cayley graph is called \emph{normal} if its generating set is closed under conjugation). Ever since this seminal work, upper bounds on character ratios played a key role in the analysis of the mixing times of normal Cayley graphs over either $A_n$ or $S_n$ (see, e.g., Roichman~\cite{roichman1996upper}, M\"{u}ller and Schlage-Puchta~\cite{muller2007character}, and Larsen and Shalev~\cite{larsen2008characters}).

To avoid sign issues we restrict our attention to mixing times of normal random walks on the alternating group $A_n$, but our result extends to $S_n$. 
Let $\mu$ be a probability measure on the alternating group. We associate to $\mu$ the function $f(\sigma) = \frac{n!}{2}\mu(\sigma)$ that satisfies $\|f\|_1 = 1.$ The measure $\mu$ corresponds to a (right-invariant) random walk, where a permutation $\sigma\in A_n$ walks to $\tau \sigma$ for a random $\tau \sim \mu.$ The function corresponding to $\ell$ steps of the random walk is given by the iterated self convolution $f^{* \ell},$ where the convolution is given by 
\[
f*g(\tau) = \mathbb{E}_{\sigma \sim A_n}[f(\sigma^{-1})g(\sigma\tau)]. 
\]
The $(\epsilon, L^p)$-\emph{mixing time} of the random walk is the minimal $\ell$ such that $\|f^{* \ell}-1\|_{p}<\epsilon.$ It is well known (and follows from Young's convolution inequality) that the $L^p$-mixing times are monotonically increasing with respect to $p$ and that the $L^{\infty}$-mixing time is at most twice the $L^2$-mixing time. It is therefore customary to restrict the attention to the case where $p$ is either $1$ or $2.$ We remark that in the special case where $\mu$ is the uniform measure on a set $A$, we have $f=\frac{n!1_A}{2|A|}$ and the random walk corresponds to the simple random walk on the Cayley graph $\mathrm{Cay}(A_n,A)$. For brevity, we sometimes refer to the mixing times of this random walk as the mixing time of $A$.

Since we focus on uniform measures, we can use the following notion of density:
\begin{definition}\label{def:density}
    Given a finite group $G$, the \emph{density} of a subset $A \subseteq G$ is $\mu(A) := \frac{|A|}{|G|}$.
\end{definition}

M\"{u}ller and Schlage-Puchta \cite{muller2007character} computed the mixing time of $\mathrm{Cay}(A_n, \sigma^{S_n})$ up to a constant factor in terms of the number of fixed points of $\sigma$, via character theoretic methods. This led Larsen and Shalev \cite{larsen2008characters} to investigate the more refined problem of characterizing conjugacy classes of $S_n$ with mixing times $\le \ell$.
When $\ell$ is fixed and $n$ goes to infinity, Larsen and Shalev \cite{larsen2008characters} provided a satisfying solution to this problem in terms of very complicated parameters that they introduced. Then, to simplify their result, they upper bounded the mixing time in terms of simpler functions of $\sigma$. One parameter of particular interest is the density $\mu(\sigma^{S_n})$. This related their bound to a long line of works in the theory of growth in groups, which aims to find conditions implying that the sizes of the sets $A^i:=\{a_1\cdots a_i: a_1,\ldots, a_i\in A\}$ grow rapidly with $i$ (see the survey~\cite{helfgott2015growth} for more details).

In particular, Larsen and Shalev posed the following problem:
\begin{problem}[\cite{larsen2008characters}]
    Given positive integers $\ell, n$, what is the largest possible density of an even normal (i.e., conjugacy-closed) set $A$ of $S_n$ which has $(1/e,L^1)$-mixing time $> \ell$?
\end{problem}

The restriction to conjugacy-closed sets is natural: for general subsets, large size alone does not guarantee rapid mixing. For instance, the set of permutations fixing $1$ has infinite mixing time despite having very large size.

Suppose that $n$ is sufficiently large and let $m = \lfloor n^{1-1/\ell} \rfloor$ for some integer $\ell$, and consider the conjugacy class $A$ of all permutations with $m$ fixed points and a single $(n- m)$-cycle (if the conjugacy class $A$ is odd, replace $m$ with $m+1$). M\"{u}ller and Schlage-Puchta \cite{muller2007character} showed that the mixing time of $A$ is larger than $\ell$. The density of $A$ in $A_n$ is given by: 
\[\mu(A) = \frac{2|A|}{n!} = \frac{2}{m!(n-m)} \ge n^{-Cn^{1-1/\ell}}\]
for some constant $C>0$. 
Larsen and Shalev were fairly close to showing that the above construction is essentially optimal by showing the following. 
\begin{thm}[{\cite[Theorem 6.8]{larsen2008characters}}]
For each $\epsilon>0$ there exists $n_0,$ such that the following holds. Let $\ell$ be an integer and let $n>n_0.$ Suppose that $A \subseteq A_n$ is a normal (i.e., conjugacy-closed) subset of $S_n$ of density $\ge n^{-n^{1- 1/\ell - \epsilon}}$, then the $(\epsilon, L^2)$ mixing time of $A$ is at most $\ell.$  
\end{thm}

Due to the $\epsilon$ in the double exponent, Larsen and Shalev's result holds only when $\ell$ is constant. We improve their result and produce an essentially optimal result when $\ell$ is up to logarithmic in $n$. Moreover, our result is applicable to all normal subsets of $A_n$ and not only to even normal subsets of $S_n$.

\begin{thm}\label{thm:Simplified mixing time}
For each $\epsilon>0$ there exists $c>0,$ such that the following holds. Let $n\in \NN$ and let $1 < \ell<\frac{c\log n}{\log\log n}$ be integers.
Suppose that $A \subseteq A_n$ is a normal subset of $A_n$ of density $\ge n^{-cn^{1- 1/\ell}}$, then the $(\epsilon, L^2)$ mixing time of $A$ is at most $\ell.$  
\end{thm}

Our result follows from the following more general version, which is applicable for arbitrary class functions and not only for indicators of sets.
\begin{thm}\label{thm:mixing time results}
For each $\epsilon>0$ there exist $c,C>0,$ such that the following holds. Let $f_1,\ldots f_\ell$ be nonnegative class functions over $A_n$ whose $L^1$-norm is 1. Let $\alpha_1, \ldots \alpha_\ell \in \left(0, 1-\frac{C\log \log n}{\log n}\right)$ be positive numbers with $\alpha_1 + \cdots + \alpha_\ell \le \ell-1$.
Suppose that $\|f_i\|_2 \le n^{c\alpha_in^{\alpha_i}}$ for all $i$. Then 
\[
\|f_1*\cdots * f_\ell - 1 \|_2 < \epsilon.
\]
\end{thm}

We also deduce the following interesting dichotomy. For each normal random walk on $A_n$ that starts at the identity we either have an unproportional probability of returning to the same place after two steps or we are mixed. Since class functions on $A_n$ may have $f(\sigma) \ne f(\sigma^{-1})$, we must assume that $f$ is \emph{symmetric}, i.e., that $f(\sigma) = f(\sigma^{-1}).$

\begin{cor}\label{cor:two steps mixing}
    For each $\epsilon>0$ there exists $c>0$, such that the following holds. Let $f\colon A_n \to \left[0,\infty\right)$ be a symmetric class function having $L^1$-norm 1. Suppose 
    $f*f(1)<n^{c\sqrt{n}}.$
    Then $\|f*f-1\|_2<\epsilon$.
\end{cor}

In Section \ref{sec:other applications} we present other applications of our techniques to the study of diameters, products of normal sets, and Kronecker coefficients.

\remove{
}

\subsection{Upper bounds on character ratios}
\label{subsec:intro character ratios}
In their seminal result, Larsen and Shalev \cite{larsen2008characters} established a general, complicated bound on the character ratios $\frac{\chi(\sigma)}{\chi(1)}$. Then, they derived the following bound in terms of the number of cycles of $\sigma$. 

\begin{thm}[{\cite[Theorem 1.4]{larsen2008characters}}, restated]
For all $\epsilon>0$ there exists $n_0$, such that if $n>n_0,$ $\alpha\in (0,1-\epsilon)$ and a permutation $\sigma$ has at most   
$n^{\alpha}$ cycles, then 
\[ \left|\frac{\chi(\sigma)}{\chi(1)} \right|\le \chi(1)^{\alpha -1+\epsilon}\] 
for all irreducible characters $\chi.$
\end{thm}

The result of Larsen and Shalev is sharp up to the $\epsilon$ in the exponent. However, this $\epsilon$ in the exponent prevents their result from providing meaningful estimates when $\sigma$ has more than $n^{1-o(1)}$ cycles. We get rid of the $\epsilon$, thereby improving the bound and getting effective results for permutations having up to $\frac{n}{\log^{O(1)} n}$ cycles. 

\begin{thm}\label{thm: few cycles}
There exist absolute constants $c,C>0$ such that the following holds. Let $\alpha \in(\frac{1}{c\log n},1-\frac{C\log\log n}{\log n})$. Suppose that $\sigma\in S_n$ has at most $c\alpha n^{\alpha}$ cycles. Then 
\[
\left|\frac{\chi(\sigma)}{\chi(1)}\right|\le \chi(1)^{\alpha -1}
\]
for all irreducible characters $\chi$.
\end{thm}

Clearly, this bound is not universally tight for all permutations, evident from the inherent variability in character values among permutations sharing the same cycle count. Nevertheless, for numerous combinations of $\chi$ and the cycle count of $\sigma$, it turns out that Theorem~\ref{thm: few cycles} represents approximately the best achievable bound based solely on the number of cycles:
\begin{thm}\label{thm: lower bound few cycles}
    For all $0 < \epsilon < 1$ there exist $n_0 \in \NN$ and $c > 0$ such that the following holds for all $n > n_0$. Let $\alpha\in(0,1-\frac{\log\log n}{\log n})$ and denote $\ell=\lceil\alpha n^{\alpha}\rceil$. Then there exists a permutation $\sigma \in S_n$ with $\ell$ cycles, such that
\[
\left|\frac{\chi(\sigma)}{\chi(1)}\right|\ge\chi(1)^{\alpha-\epsilon-1}.
\]
for every $d < \min(\frac{\ell}{200},cn^\epsilon)$ and every irreducible character $\chi = \chi_\lambda$ with $\lambda_1=n-d$.
\end{thm}

The notion of level of a character will play an important role in our paper. Let $\lambda=(\lambda_1, \lambda_2, \cdots, \lambda_\ell) \vdash n$ be a partition of $n$. We make use of the convention that $\lambda_1\ge \ldots \ge \lambda_{\ell}$ and we denote by $\lambda'$ the partition conjugate to $\lambda$ obtained by flipping the corresponding Young diagram. The \emph{level} of a character $\chi_{\lambda}$ of $S_n$ is the minimum between $n-\lambda_1$ and $n-\lambda_1'$.

We derive Theorem \ref{thm: few cycles} from Theorem \ref{thm: actual upper bound on character ratio} below, which shows that the character ratios of level $d$ representations decay essentially like $n^{d(\alpha - 1)}.$ This fits a philosophy of Gurevich and Howe \cite{gurevich2021harmonic}, which states that the main property of a character that governs the order of magnitude of its ratios is its level. Theorem \ref{thm: actual upper bound on character ratio} also demonstrates the tight connection between the theory of Boolean functions and the character theory of the symmetric group. Arguably, the theory of the Boolean cube is centered around the so called \emph{noise operator}. One of the key properties of the noise operator is the fact that its eigenvalues decay exponentially with respect to a certain parameter called degree, which is analogous to the notion of level for characters. The analogy between these notions was first observed by Ellis, Friedgut, and Pilpel~\cite{ellis2011intersecting} and many works in the theory of Boolean function are devoted to studying the analogy between Boolean functions and the harmonic analysis of the symmetric group (see e.g. \cite{dafni2021complexity, ellis2015stability}).
Given a permutation $\sigma \in S_n$, an analog to the noise operator on the Boolean cube is the operator $T_\sigma$ corresponding to the simple random walk on the Cayley graph $\mathrm{Cay}(A_n,A)$, defined by 
\[
f(\tau)\mapsto \mathbb{E}_{\pi \sim \sigma^{S_n}}[f(\pi \tau)]. 
\]
The operator $T_{\sigma}$ has the irreducible characters of $S_n$ as eigenvectors, where each character $\chi$ corresponds to the eigenvalue $\frac{\chi \left(\sigma\right)}{\chi(1)}.$ In this perspective, the following theorem shows that for many permutations $\sigma$, $T_{\sigma}$ exhibits exponential decay properties similar to the noise operator. 
\remove{
then $\frac{\chi(\sigma)}{\chi(1)} \le n^{\alpha -1}.$ 
For a set $A\subseteq S_n$ we call $\frac{A}{n!}$ the \emph{density} of $A$. By the centralizer theorem, the density of $^{S_n}\sigma$ is given by the size of the centralizer of $\sigma$.

It turns out that we have the following 

We prove the following bound  
\begin{thm}
For each $\zeta>0$ there exists $c>0,$ such that the following holds. Let $\alpha \in(0,1-\zeta)$ and let $\sigma\in S_n$ have a centralizer of size $\le e^{c n^{\alpha}\log n} $ with $\alpha \in(0, 1-\zeta).$ 
Then for every character $\chi$ of $S_n$ of level $d<n^{1-2\zeta}$ we have 
\[\left|\frac{\chi(\sigma)}{\chi(1)}\right| \le n^{d(\alpha -1)}.\]
\end{thm}
}
\begin{thm}\label{thm: actual upper bound on character ratio}
\remove{There exists an absolute constant $c>0,$ such that the following holds. Let $\alpha \in(0,1)$ and let 
$\sigma\in S_n$ be a permutation whose conjugacy class has density $\ge n^{-2c \alpha n^{\alpha}}$.
Then for every character $\chi$ of $S_n$ of level $d \le c \alpha n^{\alpha}\log n$ we have 
\[\left|\frac{\chi(\sigma)}{\chi(1)}\right| \le n^{d(\alpha-1)}\left(\frac{\alpha \log n}{\log(2c n^{\alpha}\log n) -\log d}\right)^d.\]}
There exists an absolute constant $C>0,$ such that the following holds for all $\alpha \in(0,1)$. Let 
$\sigma\in S_n$ be a permutation whose conjugacy class has density $\ge n^{-2 \alpha n^{\alpha}}$ (see Definition~\ref{def:density}).
Then for every character $\chi$ of $S_n$ of level $d \le \alpha n^{\alpha}\log n$ we have 
\[
\left| \frac{\chi(\sigma)}{\chi(1)}\right|\le  n^{d(\alpha -1)}\left(\frac{C \alpha \log n}{\log (\alpha n^{\alpha} \log n)-\log d}\right)^{d}.
\]
\end{thm}

It turns out that this bound is also very close to optimal:
\begin{thm}\label{thm: lower bound to actual upper bound on character ratio}
There exists an absolute constant $c>0,$ such that for every $\alpha \in(0,1)$ there exists $n_0 \in \NN$ such that the following holds for every $n \ge n_0$. There exists a permutation $\sigma\in S_n$ whose conjugacy class has density $\ge n^{-2 \alpha n^{\alpha}}$, such that for every character $\chi$ of $S_n$ of level $d \le 0.5\alpha n^{\alpha}$ we have 
\[
\left| \frac{\chi(\sigma)}{\chi(1)}\right|\ge  n^{d(\alpha -1)}(c\alpha)^{d}.
\]
\end{thm}

\subsection{Computing the asymptotics of character norms}

We now switch to a more analytic perspective, and study bounds on the $L^q$-norms of the characters. We will mainly be interested in the linear space of class functions, i.e., the functions that are constant on each conjugacy class. The functions on the symmetric group $S_n$ are equipped with the standard inner product 
\[\langle f, g \rangle =\mathbb{E}_{\sigma\sim S_n}[f(\sigma)g(\sigma)] := \frac{1}{n!}\sum_{\sigma \in S_n}f(\sigma)g(\sigma),\]
where the notation $\sigma \sim A$ denotes that $\sigma$ is chosen uniformly out of $A$.
We also use the $L^p$ norms 
\[\|f\|_p = \mathbb{E}^{1/p}[|f|^p].\]
The irreducible characters of the symmetric group form an orthonormal basis for the space of class functions with respect to the standard inner product. 
The coefficients 
\[
\hat{f}(\chi):= \langle f, \chi \rangle
\]
are known as the (nonabelian) \emph{Fourier coefficients} of $f$. Unlike the abelian setting where the characters take values in the unit circle and their $q$-norms are always 1, in the nonabelian setting estimating the $q$-norms of the characters is highly non-trivial. On the other hand, it may reveal lots of precious information about the character ratios, the conjugacy classes of the group and its representations, as was demonstrated in the breakthrough result of Guralnick, Larsen, and Tiep \cite{guralnick2020character} for the case of groups of Lie type. Therefore, consider the following problem.

\begin{problem}
Given an irreducible character $\chi$ of the symmetric group and $q\ge 1$, find the best possible estimate on $\|\chi\|_q.$
\end{problem}

  When $q$ is an even integer, the $q$-norms of the characters have a simple representation theoretic interpretation. The $q$-norm of the character corresponding to a representation $V$ is the $q$th root of the multiplicity of the trivial representation inside the representation $V^{\otimes q}$. \remove{While this representation theoretic information can definitely be of use\footnote{It can be used to deduce that if $\pi_1$ is a sub-representation of $\pi_2,$ then the $q$-norm of the trace of $\pi_1$ is at most the $q$-norm of the trace of $\pi_2$.}, it turns out that the problem of decomposing tensor products into irreducible characters is notoriously difficult.} We prove the following upper bound on the $L^q$-norms of the characters. 

\begin{thm}\label{thm:main}
There exists an absolute constant $C>0,$ such that the following holds. 
Let $d$ be a positive integer, and let $q\ge 2$. Let $\chi$ be a character of level $d.$
Then 
\[
\| \chi \|_q \le \left( \frac{Cq}{\log q} \right)^d\left(\frac{d^d \chi(1)}{n^d}\right)^{1-2/q}. 
\]
\end{thm}
Note that when $q<2$, we always have $\|\chi\|_q\le \|\chi\|_2 = 1$.

Theorem \ref{thm:main} is our main tool in this work, and it is applied in virtually all our other results, including the character bounds of Theorems \ref{thm: few cycles} and \ref{thm: actual upper bound on character ratio}. 

Our result is sharp up to the value of the implicit constant $C$ unless $d$ is tremendously larger with respect to $q$, as implied in the following matching lower bound.

\begin{thm}\label{thm:matching lower bound for main}
There exists an absolute constant $c>0$, such that the following holds. Let $q\ge 2$ and let $d < \min(e^{cq},\frac{n}{q+1})$ be an integer. Let $\chi$ be a character of level $d$. Then  
\[
\| \chi\|_q \ge \left( \frac{cq}{\log (qd)} \right)^d\left(\frac{d^d \chi(1)}{n^d}\right)^{1-2/q}. 
\]
\end{thm}

One of the common applications of upper bounds on the $q$-norms of random variables $X$ is the following upper bounds on the probability that $X$ is large:
\[ 
\Pr[|X|>t]<\frac{\mathbb{E}[|X|^q]}{t^q}.
\]
In the context of the symmetric group, this will allow us to deduce an upper bound on $\chi(\sigma)$ for each permutation $\sigma$ that has few cycles (see Theorem \ref{thm: few cycles} below). This might seem surprising from a probabilistic point of view, as the $q$-norm of a function $f$ typically does not provide tight upper bounds for the value of $f(x)$ at a specific point $x$. The reason why we may deduce good bounds in the symmetric group is based on two facts. Firstly, the characters of the symmetric group are constant on each conjugacy class. Secondly, the conjugacy classes of the symmetric group tend to be huge. Therefore, a bound of the form  $\Pr[\chi >t] \le p(t,\chi)$ implies that $|\chi(\sigma)| \le t$ whenever the conjugacy class of $\sigma$ has more than $p(t,\chi)\cdot n!$ elements. This simple idea will play a central role in the proof of Theorem \ref{thm: few cycles}. 

\subsubsection*{The special case $d=1$}
To get a better intuition for our bound one may plug in the partition $\lambda = (n-1,1)$ and the character $\chi_{\lambda}.$ This character has level $1$ and $\chi(1)= n-1.$  Our result therefore yields 
\[ \|\chi\|_q = \Theta\left(\frac{q}{\log q} \right).
\]
The value of $\chi$ on a permutation $\sigma$ is the number of fixed points of $\sigma$ subtracted by one.
Therefore, our bound corresponds to the fact that the number of fixed points of a random permutation is well approximated by a Poisson-distributed variable $X\sim \text{Pois}(1)$ with a mean of $1$ (as shown by Montmort in 1708). Indeed, the Poisson distribution is known to satisfy $\|X\|_q=\Theta(\frac{q}{\log q})$ for $q \ge 2$.
 
\subsection{Methods: Globalness and hypercontractivity}
\label{subsec:globalness definitions}
Our approach to obtaining character bounds is to transcend the realms of class functions and view the low level characters merely as special cases of low degree functions on the symmetric group. 
The functions on the symmetric group of the form 
$\sigma \mapsto 1_{\sigma (i) =j}$
are known as the \emph{dictators} 
and are denoted by $x_{i\to j}$. The \emph{degree} of a function $f$ on the symmetric group is the minimal $d,$
such that there exists a multivariate polynomial $P$ of degree $d$ in $n^2$ variables indexed by $\{(i,j)\}_{i,j\in \{1,\dots,n\}}$, such that $f(\sigma) = P(x_{i\to j}(\sigma))$ for all $\sigma$. 
For example the number of fixed points is a function of degree $1$ and the number of 2-cycles can be easily seen to be of degree $2$. It turns out that up to multiplication by sign, level $d$ characters of the symmetric group are functions of degree $d$ (See Ellis, Friedgut, and Pilpel \cite[ Theorem 7]{ellis2011intersecting}).
\subsubsection*{Hypercontractivity for global functions}
The tool known as `hypercontractivity for global functions' was recently developed by Keevash, Long, Lifshitz, and Minzer \cite{keevash2021global}, who were inspired by the work of Khot, Minzer, and Safra \cite{khot2018pseudorandom}. They proved a hypercontractive inequality and used it to show that low degree functions over independent random variables have small $q$-norms, provided that they satisfy a certain requirement known as globalness.  Filmus, Kindler, Lifshitz, and Minzer~\cite{filmus2020hypercontractivity} were then able to obtain somewhat similar hypercontractivity for global functions in the symmetric group, where polynomials in the dictators replace the roles of polynomials in independent random variables. The lack of independence makes the problem significantly harder in the symmetric group. Recently, Keevash and Lifshitz \cite{keevash2023sharp}, building on the work of Keller, Lifshitz, and Marcus~\cite{Keller2023sharp}, obtained a sharp hypercontractivity theorem for global functions on the symmetric group. This allowed them to provide an upper bound on the $q$-norms of low degree functions in terms of an analogue notion of globalness, which we shall now describe.  
 \subsubsection*{Globalness in the symmetric group}
Let $m\le n$ and let $I,J$ be $m$-tuples each having distinct coordinates. 
We write $U_{I\to J}$ for the set of permutations sending the elements of $I$ to the elements of $J$ in the same order. We denote the restriction of $f$ 
to $U_{I\to J}$ as $f_{I\to J}$. When we write $\|f_{I\to J}\|_2$, we are referring to the 2-norm within the set $U_{I\to J}$, i.e.,
\[\|f_{I\to J}\|_2^2=\frac{1}{(n-m)!}\sum_{\sigma\in U_{I\to J}}f(\sigma)^2.\]
We say that a function $f$ is $(r,\gamma)$-\emph{global} if for all $m$ and all pairs of $m$-tuples of distinct coordinates $I,J$ we have $\|f_{I\to J}\|_2 \le r^{m}\gamma$. 

The main ingredient in our proof for the bound on character norms (Theorem \ref{thm:main}) is the following proposition, which may also be of independent interest. 
\begin{proposition}\label{prop:characters are global}
There exists an absolute constant $C>0$ such that the following holds. Let $d$ be a positive integer and let $\chi$ be a character of level $d$. Set  $\gamma = \left(\frac{Cd}{n}\right)^d\chi(1)$. Then $\chi$ is $(2, \gamma)$-global. 
\end{proposition}

Before moving on to applications, we would like to express the hope that the theory of degrees, globalness, and hypercontractivity, is generalizable to various other transitive permutation groups. Some evidence for the generality of the theory can be derived from the work of Ellis, Kindler, Lifshitz, and Minzer \cite{Ellis2023product} who proved hypercontractivity for many compact Lie groups. 
Additionally, Evra, Kindler, and Lifshitz \cite{Evra2023Hypercontractivity} proved a hypercontractivity theorem for general linear groups over finite fields. 
While their current results may not yet yield new character bounds, we see promising potential in the hypercontractive approach that could lead to successful outcomes in the future.
\subsection*{Structure of the paper}
In Section \ref{sec:other applications}, we present the applications of our upper bound on character norms (Theorem \ref{thm:main}) to Kronecker coefficients, upper bounds on Fourier coefficients, and product mixing. Section \ref{sec:proof of main theorem} is devoted to proving Theorem \ref{thm:main}. Section \ref{sec:converse} establishes its sharpness by proving a matching lower bound (Theorem \ref{thm:matching lower bound for main}). In Section \ref{sec:Kronecker} we apply Theorem \ref{thm:main} to quickly deduce upper bounds for the Kronecker coefficients. In Section \ref{sec:Fourier} we first upper bound Fourier coefficients of class functions (Theorem \ref{thm:Upper bounds on Fourier coefficients}), and then simplify it to bound character ratios (Theorem \ref{thm: few cycles}). In Section \ref{sec:mixing times} we deduce applications for mixing times. Finally, in Section \ref{sec:product mixing} we prove applications to product mixing of normal sets.

\subsection*{Acknowledgement}
We would like to thank Gil Kalai for encouraging us to connect the theory of hypercontractivity for global functions with representation theory and Peter Sarnak for encouraging us to find connections to mixing times over the symmetric group. We would also like to thank Yotam Shomroni, Shai Evra, Dor Minzer, Guy Kindler, Ohad Sheinfeld, Doron Puder, Peter Keevash, Gady Kozma, and Yuval Filmus for many helpful suggestions. Finally, N.\ L.\ would like to thank the Simons Institute for the Theory of Computing for hosting him while the research was conducted.

\section{Other applications of Theorem \ref{thm:main}}\label{sec:other applications}
In this section we present various applications of the upper bound to character norms introduced in Theorem \ref{thm:main}. These applications include Kronecker coefficients, upper bounds on Fourier coefficients, mixing times, and product mixing of normal sets.

\subsection{Kronecker coefficients}
The most obvious application of our work is related to the study of the Kronecker coefficients. It turns out that Theorem~\ref{thm:main} already implies bounds on Kronecker coefficients that are close to the best known.

Given a group, the first goal of representation theory is to determine the set of irreducible representations for the group and to compute their dimension. Having done that, the remaining task is to understand how naturally occurring representations decompose into irreducible components. One such representation is the tensor product of two irreducible representations. For compact groups, such as the unitary group, the task is well-established due to the Littlewood-Richardson rule. However, in the symmetric group this problem is much more difficult. In fact, from a computational standpoint it is impossible to solve in general. 

Given $\lambda, \mu,\nu \vdash n$, the multiplicity of the Specht module $V_{\nu}$ inside the tensor product $V_{\lambda}\otimes V_{\mu}$ is known as the \emph{Kronecker coefficient} corresponding to $\lambda,\mu$ and $\nu$. It is given by the character formula: 
\[ g(\lambda ,\mu, \nu):=\mathbb{E}[\chi_{\lambda}\chi_{\mu}\chi_{\nu}] = \langle \chi_{\lambda}\chi_{\mu}, \chi_{\nu}\rangle. \]  
The Kronecker coefficients are $\#\text{P}$-hard to compute, and $\text{NP}$ hard to decide whether they are zero or not. Therefore, the main direction of research concerning the Kronecker coefficients is computing them in the special cases where it is possible or deriving general bounds. The reader is referred to Panova's survey \cite{panova2023complexity} for additional information on the Kronecker coefficients.

\subsubsection*{Known upper bounds}
The following dimension-based inequality is immediate:
\begin{equation}\label{eqn:trivial kronecker bound}
    g(\lambda, \mu, \nu) \le \frac{\chi_{\lambda}(1)\chi_{\mu}(1)}{\chi_{\nu}(1)},
\end{equation}
which also implies $g(\lambda, \mu, \nu) \le \min(\chi_{\lambda}(1),\chi_{\mu}(1),\chi_{\nu}(1))$ (see, e.g.,~\cite{pak2019largest}). Until recently, almost no other general bound was known, and most of the works focused on analyzing specific families of partitions, including hook shapes~\cite{remmel1989formula, ballantine2005combinatorial} and two-row shapes~\cite{remmel1994kronecker, rosas2001kronecker}. In recent years, however, several works have introduced general upper bounds on Kronecker coefficients, as functions of various parameters of the partitions. For example, Pak and Panova~\cite{pak2020bounds} bounded $g(\lambda, \mu, \nu)$ as a function of the lengths (i.e., the number of rows) of the partitions. In a subsequent work, they also bounded $g(\lambda, \mu, \nu)$ as a function of the sizes of the corresponding Durfee squares~\cite{pak2023durfee} (see therein for details).

We are not aware of any explicit general upper bound that depends only on the levels of the partitions. However, the following bound may be derived as a consequence of a work of Briand, Orellana, and Rosa~\cite{briand2011stability} on convergence rates of stable Kronecker coefficients:
\begin{thm}[Corollary of \cite{briand2011stability}] \label{thm:known kronecker bound}
Let $d_1\le d_2$ be positive integers. Let $\lambda,\mu \vdash n$ be partitions of levels $d_1$ and $d_2$ respectively, and let $\nu \vdash n$ be some partition. Then
\[
    g(\lambda, \mu, \nu) \le 4^{d_1} d_{2}^{d_{1}}d_{1}^{-d_{1}/2}.
\]
\end{thm}


Indeed, the stability result~\cite[Theorem 1.2]{briand2011stability} implies that for $n>4d_2$, removing $n-4d_2$ cells from the first parts of the partitions $\lambda, \mu$ and $\nu$ does not affect the corresponding Kronecker coefficient. Thus, we may assume $n\le 4d_2$. Combining the trivial bound $g(\lambda, \mu, \nu) \le \chi_{\lambda}(1)$ with the dimension estimate in Corollary~\ref{cor: dimension of characters n over sqrt d} under this assumption yields Theorem~\ref{thm:known kronecker bound}. (Slightly sharper bounds may be obtained from~\cite{briand2011stability}.)

\subsubsection*{Our results}
We present three upper bounds on the Kronecker coefficients in terms of the levels of the partitions, all of which follow immediately from Theorem~\ref{thm:main}. These bounds are slightly weaker than Theorem~\ref{thm:known kronecker bound}. However, our proof is considerably more elementary: the only representation-theoretic ingredient we use is Young's branching rule.

First, we may apply the upper bound 
\[ \mathbb{E}[fgh]\le \|f\|_3\|g\|_3\|h\|_3\]
implied by the generalized H\"older's inequality to obtain the following immediate corollary of Theorem~\ref{thm:main}. 
 
\begin{cor}\label{cor: Kronecker coefficients}
There exist absolute constants $C,\bar{C}>0,$ such that the following holds. Let $d_1\le d_2\le d_3$ be positive integers. Let $\lambda, \mu$ and $\nu$ be partitions of levels $d_1,d_2$ and $d_3$ respectively.
Then the Kronecker coefficients satisfy
\[
g(\lambda, \mu, \nu) \le \left(\frac{C^{d_3}d_1^{d_1}d_2^{d_2}d_3^{d_3}}{n^{d_1+d_2+d_3}}\chi_{\lambda}(1)\chi_{\mu}(1)\chi_{\nu}(1)\right)^{1/3} \le \bar{C}^{d_{3}}d_{1}^{d_{1}/6}d_{2}^{d_{2}/6}d_{3}^{d_{3}/6}.
\]
\end{cor}

Utilizing the generalized H\"{o}lder inequality $\mathbb{E}[fgh]\le \|f\|_4\|g\|_4\|h\|_2$ we obtain the following.

\begin{cor}\label{cor: Assymetric Kronecker coefficients 1}
There exist absolute constants $C,\bar{C}>0,$ such that the following holds. Let $d_1\le d_2$ be positive integers. Let $\lambda$ and $\mu $ be partitions of levels $d_1$ and $d_2$ respectively, and let $\nu$ be a partition of an arbitrary level.
Then
\[
g(\lambda, \mu, \nu) \le \left(\frac{C^{d_2}d_1^{d_1}d_2^{d_2}}{n^{d_1+d_2}}\chi_{\lambda}(1)\chi_{\mu}(1)\right)^{1/2} \le \bar{C}^{d_{2}}d_{1}^{d_{1}/4}d_{2}^{d_{2}/4}.
\]
\end{cor}

Finally, we can be slightly trickier and obtain the following bound when only $\lambda$ is assumed to have a low level. 

\begin{cor}\label{cor: Assymetric Kronecker coefficients 2}
There exists an absolute constant $C>0,$ such that the following holds. Let $d$ be an integer, let $\lambda$ be a partition of level $d$, let $\mu, \nu$ be partitions, and let $q=\frac{\log(\chi_\mu(1)\chi_{\nu}(1))}{d}.$ 
If $q \ge 2$, then
\[
g(\lambda, \mu, \nu) \le \left( \frac{Cq}{\log q}\right)^d\left(\frac{\chi_{\lambda}(1)d^d}{n^d}\right)^{1-2/q}. 
\]
\end{cor}
 
\remove{ 
Using the fact $\sum_{\nu \vdash n} g(\lambda,\lambda,\nu)^2 = \|\chi_\lambda\|_4^4$, we get a similar bound for Kronecker coefficients of the form $g(\lambda,\lambda,\nu)$.
\begin{cor}\label{cor: another bound for Kronecker coefficients}
Let $d\le n/10$ Let $\lambda \vdash n$ be a partition of level $d$.
Then there exists an absolute constant $C>0,$ such that for all $\nu \vdash n$
\[
g(\lambda, \lambda, \nu) \le \left(\frac{Cd}{n}\right)^{d}\chi_{\lambda}(1). 
\]
\end{cor}

This bound is slightly better in the case where $\chi_\mu(1) >> \chi_\lambda(1)$. We can also use the formula $\sum_{\nu \vdash n} g(\lambda,\lambda,\nu)^2 = \|\chi_\lambda\|_4^4$ to lower bound $\max_{\nu \vdash n} g(\lambda, \lambda, \nu)$ by $\left(\frac{Cd}{n \log d}\right)^{d}\chi_{\lambda}(1)$ divided by $\#\{\nu \vdash n \mid \nu_1 \ge n-2d\} \approx \#\{\nu \vdash d\} \approx e^{\tilde{c}\sqrt{d}}$.  check if this is better than naive bounds. Also check if we should use $\sum_{\nu \vdash n} g(\lambda,\mu,\nu)^2 = \|\chi_\lambda\chi_\mu\|_2^2$ or something like this to improve bounds.
}

\subsection{Upper bounds on Fourier coefficients}
In the Fourier analysis of finite abelian groups, a standard measure of pseudorandomness for a set $A$ is the smallness of the Fourier coefficients of the indicator function $1_A$. When these coefficients are small, the behavior of the set $A$ resembles that of a random set in numerous aspects.
In the nonabelian setting, a similar phenomenon arises with conjugacy classes: the pseudorandomness of the conjugacy class $\sigma^{S_n}$ corresponds to the smallness of the Fourier coefficients of $1_{\sigma^{S_n}}$. 
It turns out that all the large conjugacy classes of $A_n$ have small Fourier coefficients, leading to desirable pseudorandomness properties. It follows from the following bound on the Fourier coefficients of class functions $f$ with a small $\frac{\|f\|_2}{\|f\|_1}$.

\begin{thm}\label{thm:Upper bounds on Fourier coefficients}
There exists an absolute constant $C>0$, such that the following holds. Let $f$ be a class function with $\frac{\|f\|_2}{\|f\|_1} \le M$ for some bound $M$, and let $\chi$ be a character of $S_n$ of level $d \le  \log(M) $. 
Then 
\[
|\langle f , \chi \rangle| \le  \chi(1)\|f\|_1\left(\frac{C\log M}{n \log \left(\frac{\log M}{d}\right)} \right)^d. 
\]
\end{thm}
This result extends to $A_n$, see Theorem~\ref{thm:Upper bounds on Fourier coefficients An} below.
The proofs for our bounds on character ratios (presented in Theorems \ref{thm: few cycles} and \ref{thm: actual upper bound on character ratio}) are based on Theorem \ref{thm:Upper bounds on Fourier coefficients}, by considering the class function $1_{\sigma^{S_n}}$ in order to bound $\chi(\sigma)$.
\subsection{Product mixing}

Another application of our result concerns product mixing. We say that a group $G$ is an $\epsilon$-\emph{mixer} if for all sets $A,B$ and $C$ of density $\ge \epsilon$ (see Definition~\ref{def:density}) and independently chosen $a\sim A,b\sim B$ the probability that $ab\in C$ is within a factor of $1.01$ of $\mu(C) = \frac{|C|}{|G|}$. Gowers \cite{gowers2008quasirandom} showed that there exists an absolute constant $C$, such that $A_n$ is a $Cn^{-1/3}$-mixer. We say that $A_n$ is \emph{normally an $\epsilon$-mixer} if the same holds for all normal (i.e., conjugacy-closed) sets $A,B,C$.
We show that for normal sets Gowers' result can be improved exponentially.
\begin{thm}\label{thm:normally mixing}
There exists an absolute constant $c>0$, such that $A_n$ is normally an $n^{-cn^{1/3}}$-mixer.
\end{thm}
We also demonstrate that our result is best possible, in the sense that there exists an absolute constant $C$ such that $A_n$ is not normally an $n^{-Cn^{1/3}}$-mixer.
We show this by setting $m = \lfloor 10n^{1/3} \rfloor$and then taking $A=B=C$ to be the conjugacy class of all permutations with $m$ fixed points and one $(n- m)$-cycle (see Proposition \ref{Prop:not a mixer} below).

\remove{
 What is $C$? I tried $A=B=C$, and I think it fails: The probability that a permutation has exactly $k$ fixed points (up to a factor of $1/e$) is $\frac{\binom{n}{k}\cdot(n-k)!}{n!} = \frac{1}{k!}$ and in our case we get $(n^{n^{1/3}/3})^{-1}$. On the other hand, the probability that two sets of size $n^{1/3}$ will be the same is $\frac{1}{\binom{n}{n^{1/3}}} \approx (n^{2n^{1/3}/3})^{-1}$, which is much smaller. I think that a constant factor (like $10n^{1/3}$ fixed points) will not affect this. Maybe if $C$ requires less fixed points then the freedom in choosing the joint fixed points will give a better result?
}

Combining Theorem \ref{thm:normally mixing} with an observation of Nikolov and Pyber \cite{nikolov2011product}, we obtain the following. 

\begin{cor}\label{cor:Nikolov pyber}
    There exists an absolute constant $c>0$, such that the following holds. Suppose that $A\subseteq A_n$ is a normal set of density  $\ge n^{-cn^{1/3}},$ then $A^3=A_n$.
\end{cor}

Larsen and Tiep~\cite{larsen2023squares} showed that for each $\epsilon>0$ there exists $n_0$, such that if $n>n_0$ and $A$ is a normal subset of $A_n$ with $A=A^{-1}$ and density $\ge e^{-n^{1/4-\epsilon}}$, then $A^2 =A_n.$ We conjecture that the same holds for sets $A$ of density at least $n^{- c n}$.

Larsen and Shalev\cite{larsen2008characters} settled the case when $2$ is replaced by $4$ in a very strong sense. They showed that for every conjugacy class $A$ with at most $n/5$ fixed points we have $A^4=A_n$. Garonzi and Mar\'{o}ti~\cite{garonzi2021alternating} showed (in an equivalent formulation) that for each $\epsilon>0$ there exists $n_0$, such that if $n>n_0$ and $A,B,C,D$ are normal subsets of $A_n$ of density $\ge |A_n|^{1/2+\epsilon}$, then $ABCD = A_n$. They conjectured that there exists $c>0$, such that if $A,B,C$ are normal subsets of $A_n$ of density $\ge n^{-cn},$ then $ABC = A_n.$ This problem remains open. 

\section{Dimension and globalness of low level characters}\label{sec:proof of main theorem}

In this section we first present well known bounds for the dimensions of irreducible characters of $S_n$. We then prove Proposition \ref{prop:characters are global}, namely, that irreducible characters are global. We combine it with a result of Keevash and Lifshitz~\cite{keevash2023sharp} (Theorem \ref{thm:Hypercontractivity EKL} below) to finish the proof of Theorem \ref{thm:main} and bound character norms. Finally, we explain how to adapt our result for the alternating group.

Throughout the paper we assume some familiarity with standard notions and results from the representation theory of the symmetric group, such as Young's branching rule, standard Young tableaux and the Murnaghan--Nakayama rule. The reader is referred to~\cite{sagan2013symmetric} for more information.
In this section we show that the properties of characters with a long first row or column can be computed in terms of the properties of the smaller character obtained by deleting said row or column. 

Let $d$ be a positive integer.
We say that $\tilde{\lambda}$ is the \emph{partition of $d$ corresponding to} $\lambda$ if $\tilde{\lambda}$ is obtained from $\lambda$ by deleting $\lambda_1$ in the case where $\lambda_1 = n-d$ and by deleting $\lambda_1'$ from $\lambda'$ in the case where $\lambda_1'=n-d$. Given a character $\chi$ of a group, we denote by $\dim(\chi)$ the dimension of the corresponding representation. We also use the standard notation $[n] := \{1,\dots,n\}$ where $n$ is a positive integer.

\subsection{Dimensions of low level characters}
We make use of the following standard lower bound on the dimensions of characters and we include its proof for completeness.
\begin{lem}\label{lem: dimension of characters}
Let $d$ be a positive integer and let $\lambda\vdash n$ of level $d$. Denote $\chi := \chi_\lambda$ and $\tilde{\chi} := \chi_{\tilde{\lambda}}$.
Then $\dim(\chi) \ge \binom{n-d}{d}\dim(\tilde{\chi})$.
\end{lem}
\begin{proof}
If $n < 2d$ then the statement follows immediately, so we may assume that $n \ge 2d$. Since the character $\chi_\lambda$ and the character $\chi_{\lambda'} = \chi_\lambda \otimes \text{sign}$ obtained by transposing the corresponding Young diagram have the same dimension, we can assume without loss of generality that $\lambda_1 = n-d$. The dimension of $\chi$ is the number of ways to fill the Young diagram of shape $\lambda$ with the labels $1,\dots,n$, each appearing once, such that the rows and columns are increasing. Some of these labelings are obtained as follows: First, we put the labels $1,\ldots, d$ on the first $d$ boxes of the first row. Once we do that, the set of labels appearing below the first row can be any $S \subseteq [n]\setminus[d]$ with $|S|=d$. Moreover, for each set $S$, every valid placing of its labels below the first row corresponds to a unique valid placing of the labels of $[d]$ in the Young diagram of shape $\tilde{\lambda}$, and the number of these placings is $\dim(\tilde{\chi})$. Therefore, we have at least $\binom{n-d}{d}\dim(\tilde{\chi})$ valid ways to fill the Young diagram of shape $\lambda$.
\end{proof}

\begin{lem}\label{lem: dimension of characters converse}
Let $d$ be an integer and let $\lambda\vdash n$ of level $d$. Denote $\chi := \chi_\lambda$ and $\tilde{\chi} := \chi_{\tilde{\lambda}}$.
Then 
\[\dim(\chi) \le \binom{n}{d}\dim(\tilde{\chi}).\]
\end{lem}
\begin{proof}
Again we may assume without loss of generality that $\lambda_1 = n-d.$ 
Every filling of the Young diagram of shape $\lambda$ with increasing rows and columns is obtained through the following process: First, choose a set $S \subseteq [n]$ with $|S|=d$. Next, place the labels of $[n] \setminus S$ in the first row in increasing order. Then, place the labels of $S$ in the other rows so that re-labeling $S$ with $1,\ldots, d$ results in a valid labeling of the Young diagram of shape $\tilde{\lambda}$. The number of fillings obtained by this process is exactly $\binom{n}{d}\dim(\tilde{\chi})$. Since every valid filling can be constructed using this process, the number of valid fillings of shape $\lambda$ is at most $\binom{n}{d}\dim(\tilde{\chi})$.
\end{proof}

Lemma~\ref{lem: dimension of characters converse} may be applied to derive another lower bound which does not involve $\dim(\tilde{\chi})$.
\begin{cor}\label{cor: dimension of characters n over sqrt d}
    Let  $\chi$ be an irreducible character of $S_n$, and let us denote its level by $d$. Then
    \[\dim(\chi) \le \frac{n^d}{\sqrt{d!}}.\]
\end{cor}
\begin{proof}
    Denote by $\lambda \vdash n$ the partition associated with $\chi$. We may assume, without loss of generality, that $\lambda_1 = n-d$.
    By Lemma~\ref{lem: dimension of characters converse} we obtain that
    \[\dim(\chi) \le \binom{n}{d}\dim(\tilde{\chi}) \le \frac{n^{d}}{d!}\dim(\tilde{\chi}),\]
    where $\tilde{\chi}$ is the character associated with the partition obtained by removing the first row of $\lambda$. Since $\tilde{\chi}$ is an irreducible character of $S_d$, and the sum of squares of the dimensions of the irreducible representations of $S_d$ is $d!$, we obtain the trivial bound $\dim(\tilde{\chi}) \le \sqrt{d!}$. This completes the proof of the corollary.
\end{proof}

The following useful bound can be deduced from \cite[Claim 1, Claim 2 and Theorem 19]{ellis2011intersecting}:
\begin{thm}
\label{thm:EFP dimensions of representations}
    There exists $n_0>0$, such that the following holds. Let $n>n_0,$ $d\le n/200$ and suppose that $\chi$ is a character of $S_n$ of level at least  $d$. Then $\dim(\chi) \ge \left(\frac{n}{ed}\right)^d$. 
\end{thm}
\subsection{Globalness of low level characters}
Recall the definition of degree and the notations $U_{I\to J}$ and $f_{I\to J}$ from Section~\ref{subsec:globalness definitions}. As noted earlier, up to multiplication by sign, level $d$ characters of the symmetric group are functions of degree $d$.

We say that a function $f\colon S_n\to \mathbb{R}$ is $(r,\gamma)$-global if 
\[\|f_{I\to J}\|_2\le r^{|I|}\gamma\] 
for all $I,J$ of the same size. 

Keevash and Lifshitz \cite[Theorem 1.10]{keevash2023sharp} proved the following theorem. 
\begin{thm}\label{thm:Hypercontractivity EKL}
There exists an absolute constant $c>0$, such that the following holds. There exists a family of operators 
$\{T_{\rho}\}_{\rho \in (0,1)}$ on $L^2(S_n)$, each of which commutes with the action of $S_n$ from both sides, such that 
if $f \in L^2(S_n)$ is $(r,\gamma)$ global, $q\ge 2$ and $0<\rho \le \frac{c\log q}{qr}$, then \[ \|T_{\rho} f\|_q\le \gamma^{\frac{q-2}{q}}\|f\|_2^{2/q}.\]
Moreover, if $d< cn$ and $f$ is a function of degree $d$, then
\[ \langle T_{\rho}f,f \rangle \ge \left(c\rho\right)^d\|f\|_2^2.\]
\end{thm}

In order to apply this theorem, let us show that irreducible characters are global.
\begin{lem}
\label{lem:characters are global}
Let $d$ be a positive integer and let $\lambda\vdash n$ of level $d$. Denote $\chi := \chi_\lambda$ and $\tilde{\chi} := \chi_{\tilde{\lambda}}$.
Then $\chi$ is $(2, \tilde{\chi}(1))$-global. 
\end{lem}

In the proof of Lemma~\ref{lem:characters are global}, we will use the following result of Avni and Glazer.
\begin{lem}[{\cite[Lemma 2.7]{avni2022fourier}}]
\label{lem:glazer coset subgroup}
    Let $G$ be a finite group, let $\chi$ be an irreducible character of $G$, and let $H \le G$ be a subgroup. Then, for every $g \in G$,
    \[
        \frac{1}{|H|} \sum_{h \in H} \chi(gh) \le \langle \chi |_H, \chi_\text{triv}\rangle_H,
    \]
    where $\chi_\text{triv}$ is the trivial character of $H$.
\end{lem}

\begin{proof}[Proof of Lemma~\ref{lem:characters are global}]
    Given $m$ and two $m$-tuples $I,J$ each having distinct coordinates, we need to prove that
    \begin{equation}
    \label{eqn:proof characters are global}
        \|\chi_{I\to J}\|_2\le 2^{|I|}\tilde{\chi}(1).
    \end{equation}
    Since characters are invariant under conjugation, we may assume without loss of generality that $I=(1,\dots,m)$. The $m$-umvirate $U_{I\to I}$ is a subgroup of $S_n$ isomorphic to $S_{n-m}$. It has $U_{I\to J}$ as its coset. Choose an arbitrary permutation $g$ sending $I$ to $J$ and write $U_{I\to J} = gU_{I\to I}$.

We have 
    \[
        \|\chi_{I\to J}\|_2^2 = \frac{1}{|U_{I \to J}|} \sum_{\sigma \in U_{I \to J}} \chi^2(\sigma) = \frac{1}{(n-m)!} \sum_{\sigma \in U_{I\to I}} \chi^2(g\sigma).
    \]
   Recall that the function $\chi^2$ is a character of the symmetric group as $\chi_{\lambda}^2$ is the trace of the tensor product $V_{\lambda} \otimes V_{\lambda}$. Decomposing $\chi^2$ into irreducible characters and writing 
   $\chi^2= \sum_{i}\chi_i$
   we may apply Lemma~\ref{lem:glazer coset subgroup} to each $\chi_i$ to obtain that \[\frac{1}{(n-m)!}\sum_{\sigma \in U_{I\to I}}\chi_i(g\sigma) \le \frac{1}{(n-m)!}\sum_{\sigma \in U_{I\to I}}\chi_i(\sigma).\]
   Summing over all $i$ we obtain that $\|\chi_{I\to J}\|_2^2\le \|\chi_{I\to I}\|_2^2$. 

    Next, let us bound the $L^2(U_{I \to I})$-norm of $\chi_{I\to I}$. Clearly, it is equal to the $L^2(S_{n-m})$-norm of $\chi|_{S_{n-m}}$. We can decompose the restriction of $\chi$ to $S_{n-m}$ as a linear combination of characters and write $\chi|_{S_{n-m}} = \sum_{\mu} c_\mu \chi_\mu$, where the sum goes over all $\mu \vdash (n-m)$, and obtain that 
    \[
        \|\chi_{I\to J}\|_2^2 \le \sum_{\mu} c_\mu^2 \le \left(\sum_{\mu} c_\mu\right)^2,
    \]
    which implies that
    $\|\chi_{I\to J}\|_2 \le \sum_{\mu} c_\mu.$
    We will prove that $\sum_\mu c_\mu \le 2^m \tilde{\chi}(1)$ and our lemma will follow. 

    Recall Young's branching rule: Let $\lambda \vdash n$ and $\mu \vdash (n-m)$. The multiplicity of $\chi_\mu$ in the restriction of $\chi_\lambda$ to $S_{n-m}$ is equal to the number of valid removal processes from $\lambda$ to $\mu$, where a valid removal process is a way to remove $m$ cells from the Young diagram of shape $\lambda$ such that the resulting shape after removing the cells is $\mu$, and any intermediate shapes that are created during the removal process are also valid Young diagrams.
    
    To bound the sum of the multiplicities of all $\mu$, we count all the valid ways to remove $m$ cell from $\lambda$. We call the $i$th cell that we remove the $i$th \emph{step}. We begin counting the ways by separating them according to the set of steps in which we remove a cell from the first row of $\lambda$. There are at most $2^m$ such possible sets. Next, we fix the set $S$ of such steps. All the other steps correspond to a removal process that removes $m-|S|$ boxes from $\tilde{\lambda}$. The number of such valid ways of removing $m-|S|$ boxes is upper bounded by the number of ways to remove all the boxes of $\tilde{\lambda}$ while leaving a valid Young diagram at each step. This is exactly the dimension of $\tilde{\chi}$. We therefore obtain that $\sum_\mu c_\mu \le 2^m \tilde{\chi}(1)$, as required.
\end{proof}

\begin{proof}[Proof of Proposition~\ref{prop:characters are global}]
If $d \ge \frac{n}{10}$, then we can take $C \ge 10$ and obtain that
\[\gamma = \left(\frac{Cd}{n}\right)^d\chi(1)\ge \chi(1) = \|\chi\|_{\infty}.\]
Therefore, We may use the fact that every function $f$ is $(\|f\|_{\infty}, 2)$-global to complete the proof.

Next, suppose $d < \frac{n}{10}$. Proposition~\ref{prop:characters are global} follows as a direct consequence of Lemma~\ref{lem: dimension of characters} together with Lemma~\ref{lem:characters are global}.
\remove{We may assume without loss of generality that $d<n/10$. Otherwise, we may use the fact that every function $f$ is $(\|f\|_{\infty}, 2)$-global to complete the proof. Indeed, for $d \ge n/10$ and $C>10$ we have \[\gamma = \left(\frac{Cd}{n}\right)^d\chi(1)\ge \chi(1) = \|\chi\|_{\infty}.\]  
If $d < n/10$, then Proposition~\ref{prop:characters are global} follows as a direct consequence of Lemma~\ref{lem: dimension of characters} together with Lemma~\ref{lem:characters are global}. }
\end{proof}

We may now apply the globalness of characters to upper bound their $q$-norms.
    
\begin{proof}[Proof of Theorem~\ref{thm:main}]
The statement obviously holds for $q=2$, so we may assume that $q > 2$. Let $c>0$ be the constant from Theorem~\ref{thm:Hypercontractivity EKL}. By increasing $C$ if necessary, we may assume that $C^{\frac{1}{1-2/q}} >\frac{1}{c}$. If $d \ge cn$ then the theorem follows from the `trivial' upper bound   
\[\|\chi\|_q^q \le \chi(1)^{q-2}\|\chi\|_2^2\] and the fact that we have $\frac{C^{\frac{1}{1-2/q}}d}{n}>\frac{d}{cn}\ge 1$ and $\frac{q}{\log q} > 1$.

Next, let us assume that $d < cn$. For an arbitrary finite group $G$, it is known that if an operator $T$ on $L^2(G)$ commutes with the action of $G$ from both sides, then the irreducible characters of $G$ are eigenvectors of $T$. We therefore obtain that $T_{\rho}\chi = \langle T_{\rho}\chi, \chi\rangle \chi$ for each irreducible character $\chi.$ Let $\chi$ be a character of the symmetric group of level $d$. Our goal is to upper bound $\|\chi\|_q.$ As a character and its conjugate differ only by multiplication by the sign function, which does not affect the norm, we may assume without loss of generality that $\lambda_1=n-d.$
By \cite[Theorem 7]{ellis2011intersecting}, the characters with $\lambda_1=n-d$ are polynomials of degree $d$ in the dictators. 
By Proposition~\ref{prop:characters are global}, $\chi$ is $(2,\gamma)$-global where $\gamma = \left(\frac{C'd}{n}\right)^d\chi(1)$ for some $C'>0$. By Theorem \ref{thm:Hypercontractivity EKL} we may use the fact that $\chi$ is a polynomial of degree $d$ to obtain that there exists an absolute constant $C$, such that $\langle T_{\rho} \chi,\chi \rangle \ge \left(\frac{\rho}{C}\right)^d$, and therefore 
\[
\left(\frac{\rho}{C}\right)^d\|\chi\|_q \le \|T_{\rho}\chi \|_q \le \gamma^{1-2/q}.
\]
The theorem now follows by picking $\rho = \frac{c\log q}{2q}$ and rearranging.
\end{proof}

\subsection{Adaptation to $A_n$}\label{sec:adaptations}

Recall that the characters of $A_n$ are all obtained from characters of $S_n$ by first restricting them to $A_n$ and then decomposing them into irreducible characters when necessary.
Let us denote by $\lambda'$ the partition conjugate to $\lambda$ and denote by $\chi^{R}$ the restriction of a character $\chi$ from $S_n$ to $A_n$. We recall that when $\lambda\ne \lambda'$ the restricted character $\chi_{\lambda}^R$ is irreducible. When $\lambda =\lambda'$, the character $\chi_{\lambda}^R$ is the sum of two irreducible characters, which we denote by $\chi_{\lambda,1}$ and $\chi_{\lambda,2}$. The characters $\chi_{\lambda,1}$ and $\chi_{\lambda,2}$ have equal dimensions and therefore
\[
\chi_{\lambda,1}(1) = \chi_{\lambda,2}(1) =\frac{\chi_{\lambda}(1)}{2}. 
\]
In either case, we set the \emph{level} of the characters in the decomposition of $\chi_{\lambda}^R$ to be the level of $\chi_{\lambda}$. We also say that the characters in the decomposition of $\chi_{\lambda}^R$ are \emph{characters corresponding} to $\lambda$. The following versions of Theorem \ref{thm:EFP dimensions of representations} and Corollary \ref{cor: dimension of characters n over sqrt d} follow immediately from the above discussion. 

\begin{lem}\label{lem: dimension of characters A_n}
There exists $n_0>0$, such that the following holds. Let $n>n_0,$ $d\le n/200$ and suppose that $\chi$ is a character of $A_n$ of level at least $d$. Then $\chi(1) \ge \frac{1}{2}\left(\frac{n}{ed}\right)^d$. 
\end{lem}

\begin{cor}\label{cor: dimension of characters n over sqrt d A_n}
        Let $d$ be an integer and let $\chi$ be a character of $A_n$ of level $d$. Then
    \[\dim(\chi) \le \frac{n^d}{\sqrt{d!}}.\]
\end{cor}

Finally, we give a version of Theorem \ref{thm:main} for $A_n$. 

\begin{cor}\label{cor:of thm main}
There exists an absolute constant $C>0,$ such that the following holds. 
Let $d$ be a positive integer, and let $q\ge 2$. Let $\chi$ be a character of $A_n$ of level $d.$
Then 
\[
\| \chi \|_q \le \left( \frac{Cq}{\log q} \right)^d\left(\frac{d^d \chi(1)}{n^d}\right)^{1-2/q}. 
\]
\end{cor}
\begin{proof}
    Denote $\lambda'$ for the transpose of $\lambda$. We first consider characters of $A_n$ of type $\lambda$, where $\lambda \ne \lambda'$. For such $\chi$ we view $\chi^R$ also as a function on $S_n$ by extending it by 0 outside of $S_n.$  We now have \[\|\chi^{R}\|_{L^q(A_n)}^q = 2\|\chi^R\|_{L^q(S_n)}^{q} \le 2\|\chi\|_q^q .\]
We may now apply Theorem \ref{thm:main} to complete the proof of the case where $\lambda\ne \lambda'$. Note that we ignore the factor of $2$ since $C$ can be replaced by $2C$. 

Consider now the case where $\lambda = \lambda'$. Here the level of characters of type $\lambda$ is at least $\frac{n-1}{2}$, and therefore the theorem is trivially derived from the upper bound \[\|\chi\|_q^q\le \|\chi\|_2^2\|\chi\|_{\infty}^{q-2} =\chi(1)^{q-2}.\]  
\end{proof}

\section{Matching Lower bounds}\label{sec:converse}
In this section, we prove Theorems \ref{thm: lower bound few cycles}, \ref{thm: actual upper bound on character ratio} and \ref{thm:matching lower bound for main}, thereby establishing the sharpness of Theorems \ref{thm: few cycles}, \ref{thm: lower bound to actual upper bound on character ratio} and \ref{thm:main}, respectively. First, let us prove a simple consequence of the Murnaghan--Nakayama rule:
\begin{lem} \label{lem:murnaghan nakayama remove long cycles}
    Let $\lambda \vdash n$ be a partition with $\lambda_1 = n-d$, and let $\ell \ge d + \lambda_2$. Let $\mu = (\lambda_1 - n + \ell, \lambda_2, \ldots, \lambda_r)$ be the partition obtained from $\lambda$ by removing $n-\ell$ boxes from its first row. Let $\sigma \in S_n$ be a permutation with a cycle structure $1^\ell \alpha_1\cdots \alpha_k$ with $\alpha_i > d$. That is, $\sigma$ has $\ell$ fixed points, and the rest of the cycles of $\sigma$ have length at least $d+1$. Then we have $\chi_{\lambda}(\sigma) = \chi_{\mu}(1)$.
\end{lem}
\begin{proof}
    Recall the Murnaghan--Nakayama rule: Let $\tau \in S_{n-\alpha_k}$ be a permutation with a cycle structure $1^\ell \alpha_1\cdots \alpha_{k-1}$. Then we have
    \[
    \chi_{\lambda}(\sigma) = \sum_{\nu} (-1)^{\htt(\nu)} \chi_{\lambda\setminus \nu}(\tau),
    \]
    where the sum is over every rim hook (i.e., a connected part of the rim which can be removed to leave a proper tableau) $\nu$ of length $\alpha_k$ in a $\lambda$-tableau, $\lambda\setminus \nu$ is the tableau obtained by removing the cells of $\nu$ from $\lambda$, and $\htt(\nu) = (\text{the number of rows of $\nu$}) - 1$.
    
    Since $\lambda_1 = n - d$ and $\alpha_k > d$, every rim hook of length $\alpha_k$ must contain at least one cell of the first row. Assume, by contradiction, that a rim hook $\nu$ contains cells from both the first and the second rows. This implies that the length of $\nu$ is at least $\lambda_1 - \lambda_2+1$. Therefore, $\alpha_k \ge \lambda_1 - \lambda_2+1 > n-d-\lambda_2$. However, $\alpha_k \le n-\ell$ and $\ell \ge d+\lambda_2$, leading to a contradiction.

    Thus, there exists a unique rim hook of length $\alpha_k$ in a $\lambda$-tableau, contained within the first row of $\lambda$. Consequently, $\chi_\lambda(\sigma) = \chi_{\bar{\lambda}}(\tau)$, where $\bar{\lambda} = (\lambda_1 - \alpha_k, \lambda_2, \ldots, \lambda_r)$. The statement follows by induction on $k$.
\end{proof}

\subsection{Proof of Theorem \ref{thm: lower bound few cycles}}
\begin{proof}[\unskip\nopunct]
Let $0 < \epsilon < 1$, let $n_0 \in \NN$ and $c>0$ to be chosen later, and let $n>n_0$. Let $\alpha\in(0,1-\frac{\log\log n}{\log n})$ and denote $\ell=\lceil\alpha n^{\alpha}\rceil$. Denote $\ell' = \ell-1$, and let $\sigma \in S_n$ be a permutation with a cycle structure $1^{\ell'}(n-\ell')$. Let $d < \min(\frac{\ell}{200},cn^\epsilon)$ and let $\chi = \chi_\lambda$ be a character with $\lambda_1=n-d$. We aim to prove
\[
    \left|\frac{\chi(\sigma)}{\chi(1)}\right|\ge\chi(1)^{\alpha-\epsilon-1}.
\]

Since $\alpha < 1-\frac{\log\log n}{\log n}$, we have $\ell < n^\alpha+1 < \frac{n}{\log n}+1 < \frac{n}{3}$ for large enough $n$. Therefore we obtain that $n-\ell' >\ell > 200d > d$. Moreover, we have $\ell' \ge 200d >d+\lambda_2$. Therefore, we may apply Lemma \ref{lem:murnaghan nakayama remove long cycles} and obtain that $\chi(\sigma)=\chi_{\mu}(1)$, where $\mu \vdash \ell'$ is the partition obtained from $\lambda$ by removing $n-\ell'$ boxes from its first row. We conclude that it suffices to show that
\[
    \chi_\mu(1)\ge\chi(1)^{\alpha-\epsilon}.
\]

Since the statement clearly holds for $\alpha \le \epsilon$, we may assume that $\alpha > \epsilon$. On one hand, by Corollary \ref{cor: dimension of characters n over sqrt d}, we have
\[
    \chi(1)\le\frac{n^{d}}{\sqrt{d!}}\le\left(\frac{n\sqrt{e}}{\sqrt{d}}\right)^{d}.
\]
One the other hand, we have $d < \frac{\ell}{200}$, implying that $d \le \frac{\ell'}{200}$. In addition, since $d=\ell'-\mu_1$ and $d \le \frac{\ell'}{200}$, we deduce that the level of $\chi_\mu$ is $d$. Moreover, since $\ell > \epsilon n^\epsilon$, we may assume that $\ell'=\ell-1$ is large enough by increasing $n_0$ if necessary, so we may apply Theorem \ref{thm:EFP dimensions of representations} and deduce that
\[
    \chi_{\mu}(1)\ge\left(\frac{\ell'}{ed}\right)^{d}.
\]
As we aim to prove that $\chi_\mu(1)\ge\chi(1)^{\alpha-\epsilon}$, let us bound $\left(\frac{\chi_{\mu}(1)}{\chi(1)^{\alpha-\epsilon}}\right)^{1/d}$:
\[
    \left(\frac{\chi_{\mu}(1)}{\chi(1)^{\alpha-\epsilon}}\right)^{1/d}\ge d^{0.5(\alpha-\epsilon)-1}e^{-0.5(\alpha-\epsilon)-1}\frac{\ell'}{n^{\alpha-\epsilon}} \ge e^{-2}\frac{n^\epsilon}{d}\cdot \frac{\ell'}{n^{\alpha}}.
\]
Notably, we have $\frac{n^\epsilon}{d} > c^{-1}$. Moreover, we have $\frac{\ell}{n^\alpha} \ge \alpha > \epsilon$, implying that $\frac{\ell'}{n^\alpha} > \epsilon-n^{-\epsilon}$. For large enough $n$, we obtain $\frac{\ell'}{n^\alpha} > \frac{\epsilon}{2}$. Therefore we obtain
\[
    \left(\frac{\chi_{\mu}(1)}{\chi(1)^{\alpha-\epsilon}}\right)^{1/d} \ge 0.5e^{-2}c^{-1}\epsilon = 1
\]
for $c = 0.5e^{-2}\epsilon$, as required.
\end{proof}

\subsection{Proof of Theorem \ref{thm: lower bound to actual upper bound on character ratio}}
\begin{proof}[\unskip\nopunct]
Let $c>0$ to be chosen later, let $\alpha \in (0,1)$, and let $n \in \NN$ be a sufficiently large integer. We denote $\ell = \lceil \alpha n^\alpha \rceil$, and pick a permutation $\sigma \in S_n$ to be a permutation with a cycle structure $1^\ell (n-\ell)$. Notably, the size of the centralizer of $\sigma$ is 
\[
    \ell!\cdot(n-\ell)\le\ell^{\ell}\cdot n\le n^{\alpha n^{\alpha}+2}\le n^{2\alpha n^{\alpha}}
\]
for sufficiently large $n$. By orbit-stabilizer theorem, we obtain that the density of the conjugacy class of $\sigma$ is at least $n^{-2 \alpha n^{\alpha}}$.

Let $\chi$ be a character of level $d \le 0.5\alpha n^{\alpha}$, and denote by $\lambda \vdash n$ the corresponding partition.Let us bound its character ratios. Clearly, $d \le \ell/2$. Moreover, from $\ell = \lceil \alpha n^\alpha \rceil$ and $\alpha<1$ we obtain that $\ell<n/2$ for sufficiently large $n$, implying that $n-\ell>\ell\ge d$. Therefore, we may apply Lemma \ref{lem:murnaghan nakayama remove long cycles} and obtain $\chi(\sigma) = \chi_{\mu}(1)$, where $\mu$ is the partition obtained from $\lambda$ by removing $n-\ell$ boxes from its first row.

By Lemmas \ref{lem: dimension of characters} and \ref{lem: dimension of characters converse} we have
\[\chi(\sigma) = \chi_{\mu}(1)\ge \binom{\ell-d}{d} \chi_{\tilde{\lambda}}(1)\ge \frac{\binom{\ell-d}{d}}{\binom{n}{d}} \chi(1)\ge \left(\frac{c\ell}{n}\right)^d \chi(1)\]
for an absolute constant $c$. Therefore, we obtain
\[
    \left|\frac{\chi(\sigma)}{\chi(1)}\right|\ge\left(\frac{c\ell}{n}\right)^{d}\ge\left(\frac{c\alpha n^{\alpha}}{n}\right)^{d}=n^{d(\alpha-1)}\left(c\alpha\right)^{d}.
\]
\end{proof}

\subsection{Proof of Theorem \ref{thm:matching lower bound for main}}
First, we prove a probabilistic proposition that will be helpful
later:
\begin{proposition} \label{prop:probability no short cycles}
    Let $d$ and let $r>d$. Let $F_{r,d}$ be the event that a random permutation $\sigma\sim S_r$ has all its cycles of length $>d$. Then $\Pr[F_{r,d}]\ge \frac{1}{10d}$.
\end{proposition}
\begin{proof}
The distribution of the cycle decomposition of a random permutation in $S_r$ is identical to the distribution of a random sequence $(i_1,i_2,\dots)$ obtained by the following process: First choose a random $i_1\sim [r]$, then repeat the process recursively with $r-i_1$ in place of $r$ to obtain $(i_2,\dots)$. This follows easily from the fact that the length of the cycle of the element $r$ is uniform, and that when removing its cycle we get a uniformly random permutation on the remaining elements. 
This yields the recursive formula \[ \Pr[F_{r,d}] = \frac{1}{r} + \sum_{i=d+1}^{r-d-1}\frac{1}{r}\Pr[F_{r-i,d}].\]

We now show that $\Pr[F_{r,d}]\ge \frac{1}{10d}$ by induction on $r> d$ for a fixed $d$. For $d+1 \le r \le 10d$ the claim is obvious, since $\Pr[F_{r,d}] \ge \frac{1}{r}$ by the recursive formula. Let $r > 10d$. We may apply the induction hypothesis to obtain that
\[
    \Pr[F_{r,d}]\ge \frac{1}{r} + \sum_{i=d+1}^{r-d-1} \frac{1}{10dr} = \frac{1}{r}+\frac{r-2d-1}{10dr}>\frac{1}{10 d}.
\]
\end{proof}

\begin{proof}[Proof of Theorem \ref{thm:matching lower bound for main}]
Let $d < \min(e^{cq},\frac{n}{q+1})$ and let $\chi_{\lambda}$ be a character of level $d$. As multiplication by sign does not affect the $q$-norm we may assume without loss of generality that $\lambda_1 = n-d$. Pick some positive integer $\ell \ge 2d$ that will be determined later. Let $\sigma$ be a permutation with $\ell$ fixed points such that rest of the cycles of $\sigma$ have length at least $d+1$. Let $\mu = (\lambda_1 - n + \ell, \lambda_2, \ldots, \lambda_r)$ be obtained from $\lambda$ by removing $n-\ell$ boxes from its first row. As usual, let $\tilde{\lambda}$ be obtained from $\lambda$ by deleting its first row. Since $\lambda_2 \le d$, we may apply Lemma \ref{lem:murnaghan nakayama remove long cycles} and obtain $\chi_{\lambda}(\sigma) = \chi_{\mu}(1)$.
By Lemmas \ref{lem: dimension of characters} and \ref{lem: dimension of characters converse} we have
\[\chi_{\mu}(1)\ge \binom{\ell-d}{d} \chi_{\tilde{\lambda}}(1)\ge \frac{\binom{\ell-d}{d}}{\binom{n}{d}} \chi_{\lambda}(1)\ge \left(\frac{c\ell}{n}\right)^d \chi_{\lambda}(1)\]
for an absolute constant $c$. Let $E$ be the event that a random $\sigma\sim S_n$ has $\ell$ fixed points and all other cycles of length $>d$. Then we may 
obtain that
\[
\| \chi_{\lambda} \|_q^q \ge \Pr[E] \chi^q_\mu(1) \ge \Pr[E] \left(\frac{c\ell}{n}\right)^{dq}\chi^q_{\lambda}(1). 
\]

Let $F_{r,d}$ be the event that a random permutation $\sigma\sim S_r$ has all its cycles of length $>d$. Proposition~\ref{prop:probability no short cycles} implies that whenever $n-\ell > d$, we have $\Pr[F_{n-\ell,d}]\ge \frac{1}{10d}$. Therefore,
\begin{eqnarray*}
\Pr[E] &=& \binom{n}{\ell}\frac{1}{n(n-1)\cdots (n-\ell +1)}\Pr[F_{n-\ell,d}] \\
       &=& \frac{1}{\ell!}\Pr[F_{n-\ell,d}] \ge \frac{1}{10d \ell!} \ge \frac{1}{ (c\ell)^\ell d}
\end{eqnarray*}
for an absolute constant $c$.

Substituting $\ell=\frac{qd}{\log (qd)},$ while noting that $\ell\ge 2d$ and $n - \ell > d$ provided that $d < \min(e^{cq},\frac{n}{q+1})$, we obtain  
\begin{eqnarray*}
\|\chi_{\lambda}\|_q 
&\ge& \left(\frac{c^2qd}{\log (qd)}\right)^d\frac{\chi_\lambda(1)}{n^{d}} \\
&=& \left(\frac{c^2q}{\log(qd)}\right)^d\frac{d^d\chi_{\lambda}(1)}{n^d} \\
&\ge& \left(\frac{c'q}{\log(qd)}\right)^d \left(\frac{d^d\chi_{\lambda}(1)}{n^d}\right)^{1-2/q}
\end{eqnarray*}
for an absolute constant $c'$. The last inequality follows from the fact that $\chi_\lambda(1) \ge \binom{n-d}{d},$ which yields
$\left(\left(\frac{d}{n}\right)^d\chi_{\lambda}(1)\right)^{2/q} \ge \left(\frac{1}{3}\right)^{d}$.
\end{proof}

\section{Upper bounds on Kronecker coefficients}\label{sec:Kronecker}
\begin{proof}[Proofs of Corollaries \ref{cor: Kronecker coefficients} ,\ref{cor: Assymetric Kronecker coefficients 1}, and \ref{cor: Assymetric Kronecker coefficients 2}, ] The proofs of Corollaries \ref{cor: Kronecker coefficients} ,\ref{cor: Assymetric Kronecker coefficients 1}, and \ref{cor: Assymetric Kronecker coefficients 2} share a common framework, where we make use of Theorem \ref{thm:main}. By applying Theorem \ref{thm:main} in different ways, we can establish the validity of each statement.

Corollary \ref{cor: Kronecker coefficients} follows directly from Theorem \ref{thm:main} by applying the generalized H\"{o}lder inequality to obtain \[\mathbb{E}[\chi_{\lambda}\chi_{\mu} \chi_{\nu}] \le \|\chi_{\lambda}\|_3\|\chi_{\mu}\|_3\|\chi_{\nu}\|_3.\]
Corollary \ref{cor: Assymetric Kronecker coefficients 1} follows similarly when applying the generalized H\"{o}lder inequality with different parameters to obtain 
\[
\mathbb{E}[\chi_{\lambda}\chi_{\mu} \chi_{\nu}]\le \|\chi_{\nu}\|_2 \|\chi_\mu\|_4 \| \chi_{\lambda}\|_4.
\]
In Corollaries \ref{cor: Kronecker coefficients} and \ref{cor: Assymetric Kronecker coefficients 1} we also apply Corollary~\ref{cor: dimension of characters n over sqrt d} to obtain the second inequality.

Finally, for Corollary~\ref{cor: Assymetric Kronecker coefficients 2} we apply the generalized H\"{o}lder inequality to obtain that 
\[
\mathbb{E}[\chi_{\lambda}\chi_{\mu} \chi_{\nu}]\le \|\chi_{\lambda}\|_q\|\chi_{\mu}\|_{2/(1-1/q)}\|\chi_{\nu}\|_{2/(1-1/q)},
\]
where $q=\frac{\log(\chi_\mu(1)\chi_{\nu}(1))}{d}$.
We may upper bound $\|\chi_{\lambda}\|_{q}$ via Theorem \ref{thm:main}. In order to bound $\|\chi_{\mu}\|_{2/(1-1/q)}$, notice that
\[ \|\chi_{\mu}\|_{\frac{2}{1-1/q}}^{\frac{2}{1-1/q}} = \mathbb{E}\left[|\chi_{\mu}|^{\frac{2}{1-1/q}}\right] \le \|\chi_{\mu}\|_2^2\chi_{\mu}(1)^{\frac{2}{1-1/q}-2} = \chi_{\mu}(1)^{\frac{2}{q-1}}.\]
Therefore, $\|\chi_{\mu}\|_{\frac{2}{1-1/q}} \le \chi_{\mu}(1)^{\frac{1}{q}} \le e^d$. A similar bound holds for $\chi_{\nu}$.
Combining these inequalities completes the proof.
\end{proof}

\section{Character ratios and Fourier coefficients}\label{sec:Fourier}
In this section we prove our upper bound on Fourier coefficients of functions in terms of their $2$-norm and $1$-norm (Theorem \ref{thm:Upper bounds on Fourier coefficients}). We then deduce our upper bounds on the character ratios of conjugacy classes with few cycles (Theorem \ref{thm: few cycles}). 

\begin{proof}[Proof of Theorem \ref{thm:Upper bounds on Fourier coefficients}]
Let $q \ge 2$ to be chosen later, and let $q'$ be its H\"{o}lder conjugate. Then we have $|\langle f , \chi \rangle|\le \|f\|_{q'}\|\chi\|_q.$ We may then apply the log-convexity of $L^p$-norms to obtain that \[\|f\|_{q'}\le \|f\|_1^{1-\theta}\|f\|_2^{\theta} = \|f\|_1\left(\frac{\|f\|_2}{\|f\|_1}\right)^{\theta} \le \|f\|_1 M^\theta,\] where $1-\frac{1}{q} = \frac{1}{q'} = \frac{\theta}{2} +\frac{1-\theta}{1}.$ Rearranging, we obtain that 
$q=\frac{2}{\theta}.$ Choosing $\theta = \frac{d}{ \log M}$,
we obtain that  
$
M^{\theta} \le e^d
$
and $q=\frac{2 \log M}{d} \ge 2$. The theorem now follows by plugging in Theorem~\ref{thm:main} to obtain that 
\[\|\chi\|_q \le \left(\frac{Cq}{\log q}\right)^d \left( \frac{d^d\chi(1)}{n^d}\right)^{1-2/q}\le \left(\frac{2Cq}{\log q}\right)^d\left( \frac{d^d\chi(1)}{n^d}\right).\]
Plugging everything together we obtain that
\[\langle f, \chi \rangle\le  \chi(1)\|f\|_1\left(\frac{4e C \log M}{n \log \left(\frac{2 \log M}{d}\right)} \right)^d, \]
which completes the proof with $4eC$ instead of $C$. 
\end{proof}

\remove{
When the condition $d \le \frac{\log(\|f\|_2/\|f\|_1)}{4}$ is not satisfied, we still have the following bound:

\begin{thm}
There exists an absolute constant $C>0$, such that the following holds. Let $\chi$ be a character of $S_n$ of level $d \le n/10$ and $d \ge \frac{2\log(\|f\|_2/\|f\|_1)}{3}$. 
 Then 
\[
|\langle f , \chi \rangle| \le  \chi(1)\|f\|_1\left(\frac{Cd}{n} \right)^d. 
\]
\end{thm}
\begin{proof}
Let $q = 3$, and let $q' = \frac{3}{2}$ be its H\"{o}lder conjugate. Then we have $|\langle f , \chi \rangle|\le \|f\|_{q'}\|\chi\|_q.$ We may then apply the log-convexity of $L^p$-norms to obtain that $\|f\|_{q'}\le \|f\|_1^{1-\theta}\|f\|_2^{\theta} = \|f\|_1\left(\frac{\|f\|_2}{\|f\|_1}\right)^{\theta} ,$ where $1-\frac{1}{q} = \frac{1}{q'} = \frac{\theta}{2} +\frac{1-\theta}{1}.$ Rearranging, we obtain that 
$q=\frac{2}{\theta}$, so $\theta=\frac{2}{3}$. Since $\theta \le \frac{d}{\log(\|f\|_2/\|f\|_1)}$, 
we obtain that
$
\left(\frac{\|f\|_2}{\|f\|_1}\right)^{\theta} \le e^d.
$
The theorem now follows by plugging in Theorem~\ref{thm:main} to obtain that
\[\|\chi\|_q \le \left(\frac{Cq}{\log q}\right)^d \left( \frac{d^d\chi(1)}{n^d}\right)^{1-2/q}\le C^d\left( \frac{d^d\chi(1)}{n^d}\right).\]
Plugging everything together we obtain that 
\[\langle f, \chi \rangle\le  \chi(1)\|f\|_1\left(\frac{e Cd}{n} \right)^d, \]
which completes the proof with $eC$ instead of $C$. 
\end{proof}
}

We now deduce Theorem \ref{thm: actual upper bound on character ratio}.
\begin{proof}[Proof of Theorem \ref{thm: actual upper bound on character ratio}] Let $C$ be the conjugacy class of $\sigma,$ and write $f=\frac{1_C}{\mu(C)}$. Since \[\frac{\|f\|_2}{\|f\|_1} = \frac{1}{\sqrt{\mu(C)}} \le n^{\alpha n^{\alpha}}\] and $d \le \alpha n^{\alpha} \log n$, we may apply Theorem~\ref{thm:Upper bounds on Fourier coefficients} with $M=n^{\alpha n^{\alpha}}$ and obtain that
\[
    |\chi(\sigma)| = |\langle f, \chi \rangle| \le \chi(1) \left(\frac{C \alpha n^{\alpha} \log n}{n\left(\log (\alpha n^{\alpha} \log n)-\log d\right)} \right)^d
\]
for some constant $C>0$. This yields that 
\[
\left| \frac{\chi(\sigma)}{\chi(1)}\right|\le  n^{d(\alpha -1)}\left(\frac{C \alpha \log n}{\log (\alpha n^{\alpha} \log n)-\log d}\right)^{d}.
\]
\end{proof}
We now present several technical lemmas that will be applied later to prove Theorem \ref{thm: few cycles} and to analyze mixing times.
The following lemma simplifies Theorem \ref{thm:Upper bounds on Fourier coefficients} by breaking it down into three regimes, making it easier to apply later.
\begin{lem}\label{cor: Dividing our upper bound into regimes}
    \remove{There exist absolute constants $C,c$, such that the following holds. Let $f\colon S_n \to \mathbb{R}$ be a function normalized such that $\|f\|_1=1$ and $\|f\|_2\le e^{cn}$. Let $d$ be a natural number and let $\chi$ be a character of level $d$.}
    There exists an absolute constant $C$ such that the following holds. Let $G$ be either $A_n$ or $S_n$, and let $f\colon G \to \mathbb{R}$ be a function with $\|f\|_1=1$. Suppose that $d$ is a positive integer and let $\chi$ be a character of $G$ of level $d$.
    Then the following holds.
    \begin{enumerate}
        \item If $d\le \log^{0.9}\|f\|_2$ and $\|f\|_2>e$, then
        \[
        \frac{|\langle f, \chi \rangle|}{\chi(1)} \le \left(\frac{C\log\|f\|_2}{n\log\log\|f\|_2}\right)^d. 
        \]
        \item If $d \le \log \|f\|_2$ then 
        \[
        \frac{|\langle f, \chi \rangle|}{\chi(1)} \le \left(\frac{C\log\|f\|_2}{n}\right)^d.
        \]
        \item Finally, if $d \ge \log \|f\|_2$ then
        \[
        \frac{|\langle f, \chi \rangle|}{\chi(1)}\le \frac{\|f\|_2}{\chi(1)}.
        \]
    \end{enumerate}
    \remove{Then in the regime where $d\le \log^{0.9}\|f\|_2$  we have 
    \[
    \frac{|\langle f, \chi \rangle|}{\chi(1)} \le \left(\frac{C\log\|f\|_2}{n\log\log\|f\|_2}\right)^d. 
    \]
    In the regime where $\log^{0.9}\|f\|_2 < d \le \frac{\log \|f\|_2}{4}$ we have 
    \[ \frac{|\langle f, \chi \rangle|}{\chi(1)} \le \left(\frac{C\log\|f\|_2}{n}\right)^d \]
    Finally, in the regime where $d>\frac{1}{4}\log\|f\|_2$ we have 
    \[ \frac{|\langle f, \chi \rangle|}{\chi(1)}\le \frac{\|f\|_2}{\chi(1)}.\]}
\end{lem}
\begin{proof}
    The result for the first two regimes follows immediately from Theorem \ref{thm:Upper bounds on Fourier coefficients} with $M= \|f\|_2$. The result for the third regime follows from Cauchy--Schwarz.
    \end{proof}

In the third regime we obtain the following.
\begin{lem}\label{lem: largest degree}
For each $\epsilon>0$ there exist $c,C>0$ such that the following holds. Let $\alpha\in \left( \frac{1}{c\log n}, 1 - \frac{C\log \log n}{\log n}\right)$, and let $f\colon S_n \to \mathbb{R}$ be a function with $\|f\|_1=1$ and  \[\|f\|_2\le n^{c\alpha n^{\alpha}}.\] Let $d\ge \log\|f\|_2$ be a positive integer, and let $\chi$ be a character of level $d$. Then 
\[ \frac{\left|\langle f, \chi\rangle\right|}{\chi(1)} \le \max(\epsilon^{d}, 2^{-n^{3/5}})\chi(1)^{\alpha -1}.
\]
\end{lem}
\begin{proof}
    By decreasing $c$ if necessary, we may assume that $n>n_0,$ where $n_0$ is sufficiently large.

    Let $M= \log\|f\|_2$. Since $M<n/200$ for $n$ sufficiently large, we obtain $\chi(1) \ge \left(\frac{n}{eM}\right)^{M}$ by Theorem \ref{thm:EFP dimensions of representations}. We also have $\frac{\left|\langle f, \chi\rangle\right|}{\chi(1)} \le \frac{\|f\|_2}{\chi(1)},$ which shows that \[\frac{\left|\langle f, \chi\rangle\right|}{\chi(1)} \le \chi(1)^{\alpha -1} \frac{\|f\|_2}{\chi(1)^{\alpha}}\le \chi(1)^{\alpha-1}\frac{e^{M}}{e^{\frac{1}{2}\alpha M\log(\frac{n}{eM})}\chi(1)^{\alpha/2}}.\]
    Now 
     \[ \chi(1)^{\alpha-1}\frac{e^{M}}{e^{\frac{1}{2}\alpha M\log(\frac{n}{eM})}} \le  \chi(1)^{\alpha-1}\|f\|_2^{-\alpha/2\log\left(\frac{n}{e^{2/\alpha +1}M}\right)} \le \chi(1)^{\alpha -1}.\]
     Moreover, for $d\le n^{3/4}$ we may apply Theorem \ref{thm:EFP dimensions of representations} and obtain that \[\chi(1)^{\alpha/2} \ge \left(\frac{n}{ed}\right)^{d\alpha/2} \ge n^{d/9 \alpha} \ge \left(e^{\frac{1}{9c}}\right)^d \ge \epsilon^{-d},\]
     provided that $c$ is sufficiently small. For $d\ge n^{3/4}$, we apply Theorem \ref{thm:EFP dimensions of representations} with $n^{3/4}$ in place of $d$. We obtain that
     \[\chi(1)^{\alpha/2} \ge \left(\frac{n^{1/4}}{e}\right)^{n^{3/4}\alpha/2}\ge 2^{n^{3/5}},\]
     provided that $n$ is sufficiently large.
     This completes the proof of the lemma.
\end{proof}

\remove{
\begin{lem}\label{lem:largest degree variant}
For each $\omega>0$, there exists $n_0,C>0$
such that the following holds. Let $n>n_0$, $\alpha\in(\frac{C}{\omega\log n},1)$, and let $f\colon S_n \to \mathbb{R}$ be a function with $\|f\|_1=1$ and 
\[\|f\|_2\le e^{n^{1/3}}.\]
Let $d>\frac{C\log\|f\|_2}{\log n}$ be a natural number and let $\chi$ be a character of level $d$. Then 
\[ \frac{\left|\langle f, \chi\rangle\right|}{\chi(1)} \le \chi(1)^{\alpha -1} \|f\|_2^{-\omega}.
\]
\end{lem}
\begin{proof}
     Let $M= \frac{\omega C\log\|f\|_2}{ \alpha \log n}$. 
     We have $\chi(1) \ge \left(\frac{n}{eM}\right)^{M}\ge n^{M/2}$ 
     by Theorem \ref{thm:EFP dimensions of representations}. We also have $\frac{\left|\langle f, \chi\rangle\right|}{\chi(1)} \le \frac{\|f\|_2}{\chi(1)},$ which shows that \[\frac{\left|\langle f, \chi\rangle\right|}{\chi(1)} \le \chi(1)^{\alpha -1} \frac{\|f\|_2}{\chi(1)^{\alpha}}\le \chi(1)^{\alpha-1}\|f\|_2 n^{-\alpha M/2} \le  \chi(1)^{\alpha-1}\|f\|_2^{-100}.\]  
\end{proof}
}

We now move on to the range where $d$ is small.

\begin{lem}\label{lem: smallest degree}
 For each $\epsilon>0$, there exists $c>0$, such that the following holds. Let $f\colon S_n \to \mathbb{R}$ be a class function with $\|f\|_1=1$. Let $d$ be a positive integer and let $\chi$ be a character of level $d$. Suppose that $d \le \log^{0.9}\|f\|_2,$ and let $\alpha\in \left( \frac{1}{c\log n}, 1 - \frac{C\log \log n}{\log n}\right)$ be with  $\|f\|_2 \le n^{c\alpha n^{\alpha}}$. Then
\[ \frac{|\langle f, \chi \rangle|}{\chi(1)} \le \epsilon^{d}\chi(1)^{\alpha - 1}.\]
\end{lem}

\begin{proof}
    The statement holds for $\|f\|_2 \le e^e$, since we may apply the bound $|\hat{f}(\chi)|\le \|f\|_2$ and $\chi(1)\ge n-1$ for all characters $\chi$ of positive level.  We may therefore assume that  $\|f\|_2 > e^e$.
Write $M = \log\left((n^{c\alpha n^{\alpha}}\right)) > e.$
    By corollary \ref{cor: Dividing our upper bound into regimes} we have
    \[
    \frac{|\langle f, \chi \rangle|}{\chi(1)} \le \left(\frac{C'\log\|f\|_2}{n\log\log\|f\|_2}\right)^d \le \left(\frac{C' M}{\log M}\right)^d,
    \]
    by monotonicity of $x/\log x$ for $x>e$.
    
Since 
$\alpha>\frac{1}{c\log n}$ we have

    \[ M = c\alpha n^{\alpha} \log n \ge n^{\alpha}. \]
    This shows that $\log M\ge \alpha\log n$
and hence,     
    \[
    \frac{|\langle f, \chi \rangle|}{\chi(1)}  \le \left(\frac{C'c\alpha n^{\alpha}\log n}{\alpha n \log n}\right)^d \le \left(\epsilon   n^{\alpha-1}\right)^d, 
    \]
    provided that $c\le \frac{\epsilon}{C'}$.
    We may now use Corollary \ref{cor: dimension of characters n over sqrt d} to upper bound $\chi(1)\le n^d$, which completes the proof.
\end{proof}

Finally we move on to the remaining regime for the values of $d$.
\begin{lem}\label{lem:medium degee}
For each $\epsilon>0$ there exist $C,c>0$, such that the following holds. Let $f\colon S_n \to \mathbb{R}$ be a class function with $\|f\|_1=1$. Let $d$ be a natural number, $\alpha\in \left( \frac{1}{c\log n}, 1 - \frac{C\log \log n}{\log n}\right)$ and  $\chi$ a character of level $d$. If \[\log^{0.9}\|f\|_2 < d < \log\|f\|_2,\] and  $\|f\|_2 \le n^{c\alpha n^{\alpha}},$ then
\[ \frac{|\langle f ,\chi \rangle|}{\chi(1)} \le \epsilon^{d}\chi(1)^{\alpha - 1}.\]
\end{lem}
\begin{proof}
By decreasing $\alpha$ if necessary we may assume that either $\|f\|_2 = n^{c\alpha n^{\alpha}}$ or  $\alpha = \frac{100\log \log n}{\log n}$. By Corollary \ref{cor: Dividing our upper bound into regimes} we have  
    \[
   \frac{|\langle f ,\chi \rangle|}{\chi(1)} 
   \le \left(\frac{C'\log\|f\|_2}{n}\right)^d = \chi(1)^{\alpha-1}\left(\frac{C'\log\|f\|_2}{n\log \log \|f\|_2}\right)^d n^{(1-\alpha)d}\cdot  A,\]
   where 
   \[
   A = \left(\log \log \|f\|_2\right)^d\frac{\chi(1)^{1-\alpha}}{n^{(1-\alpha)d}}.
   \]
   Similarly to the proof of Lemma \ref{lem: smallest degree} we obtain that \[\frac{|\langle f,\chi \rangle|}{\chi(1)} \le A\epsilon'^d\chi(1)^{\alpha-1}\]
   for an $\epsilon'>0$ of our choice. 
   Therefore, our proof will be completed once we show that $A\le C''^d$ for some sufficiently large $C''$.
   
    Now, \[\log\log \|f\|_2 \le \frac{\log d}{0.9}.\] Moreover, by Corollary~\ref{cor: dimension of characters n over sqrt d}, \[\chi(1)\le \frac{n^de^d}{d^{d/2}}.\] We will soon show that $\alpha< 1-\frac{100\log \log d}{\log d}$. Combining with the hypothesis, we obtain
    \[
    A \le e^{d} \left(\frac{\log d}{0.9}\right)^d d^{-\frac{1}{2}d(1-\alpha)} \le C''^d.
    \]
    
    We now show that  $$\alpha< 1-\frac{100\log \log d}{\log d}.$$ If $\alpha <1/2$ we are done, so suppose that $\alpha \ge 1/2.$ We have $\|f\|_2 = n^{c\alpha n^{\alpha}}$. This implies that
    \[
    \log \|f\|_2 = c\alpha n^{\alpha}\log n \ge n^{\alpha} 
    \]
    and hence, 
    \[
     \log \log \|f\|_2 \ge  \frac{1}{2}\log n.
    \]
    
    On the other hand, we have 
    \[
    d >  \log^{0.9} \|f\|_2
    \] and so 
    \[
    \log d > 0.9 \log \log \|f\|_2 \ge \frac{1}{20}\log n. 
    \]
By hypothesis we have 
\[
\alpha < 1-\frac{C\log\log n}{\log n}< 1-\frac{100\log \log d}{\log d}, 
\]
provided that $C$ is sufficiently large. This completes the proof of the lemma.
\end{proof}

Combining the above lemmas yields the following theorem. 
\begin{thm}\label{thm: Useful for mixing times}
For each $\epsilon>0$ there exist  $c,C>0$ such that the following holds. Let \[\alpha \in\left(\frac{1}{c\log n}, 1-\frac{C\log \log n}{\log n}\right).\] Suppose that a class function $f\colon S_n\to \mathbb{R}$ satisfies $\|f\|_2 \le  n^{c\alpha n^\alpha}.$ Then for every $d>0$ and every character $\chi$ of level $d$ we have 
\[\frac{\left|\langle f,\chi\rangle \right|}{\chi(1)} \le \max\left(\epsilon^d, 2^{-n^{3/5}}\right) \chi(1)^{\alpha -1}.\] 
\end{thm}
\begin{proof}
The theorem follows immediately by combining Lemmas \ref{lem: largest degree}, \ref{lem: smallest degree}, and \ref{lem:medium degee}.
\end{proof}

\remove{
$n^{c\alpha n^{\alpha}}\le e^{\frac{n}{\mathrm{exp}(200/\alpha)}}$

$c\beta 2^{\beta}n^{200/\beta}\le n$
}
We are now ready to deduce our upper bound on the character ratios of conjugacy classes with few cycles. 
\begin{proof}[Proof of Theorem \ref{thm: few cycles}]
First we note that if a permutation $\sigma$ has at most $c\alpha n^{\alpha}$ cycles, then by the orbit-stabilizer theorem there exists a constant $c'=c'(c)$, such that the conjugacy class of $\sigma$ has density $\prod_{i=1}^{n}(f_i!i^{f_i})^{-1}\ge n^{-c'\alpha n^{\alpha}},$ where $f_i$ is the number of cycles of $\sigma$ of length $i$.  
Theorem \ref{thm: few cycles} now follows from Theorem \ref{thm: Useful for mixing times} with $\epsilon=1$.
\end{proof}

\subsection{Adaptation to $A_n$}\label{sec:adaptation 2}

To adapt these results to the group $A_n$ instead of $S_n$, we observe that most of the analysis applied in this section requires general arguments on finite groups. The only properties of $S_n$ we used were Theorem~\ref{thm:main}, Corollary~\ref{cor: dimension of characters n over sqrt d} and Theorem~\ref{thm:EFP dimensions of representations}. Therefore, if we replace the use of these properties with the analogous statements Corollary~\ref{cor:of thm main}, Corollary~\ref{cor: dimension of characters n over sqrt d A_n} and Lemma~\ref{lem: dimension of characters A_n}, respectively, we obtain the following adapted results for $A_n$:
\begin{thm}[Adaptation of Theorem \ref{thm:Upper bounds on Fourier coefficients}]\label{thm:Upper bounds on Fourier coefficients An}
There exists an absolute constant $C>0$, such that the following holds. Let $f$ be a class function with $\frac{\|f\|_2}{\|f\|_1} \le M$ for some bound $M$, and let $\chi$ be a character of $A_n$ of level $d \le  \log(M) $. 
Then 
\[
|\langle f , \chi \rangle| \le  \chi(1)\|f\|_1\left(\frac{C\log M}{n \log \left(\frac{\log M}{d}\right)} \right)^d. 
\]
\end{thm}

\begin{thm}[Adaptation of Theorem \ref{thm: actual upper bound on character ratio}] \label{thm: actual upper bound on character ratio An}
There exists an absolute constant $C>0,$ such that the following holds for all $\alpha \in(0,1)$. Let 
$\sigma\in A_n$ be a permutation whose conjugacy class has density $\ge n^{-2 \alpha n^{\alpha}}$.
Then for every character $\chi$ of $A_n$ of level $d \le \alpha n^{\alpha}\log n$ we have 
\[
\left| \frac{\chi(\sigma)}{\chi(1)}\right|\le  n^{d(\alpha -1)}\left(\frac{C \alpha \log n}{\log (\alpha n^{\alpha} \log n)-\log d}\right)^{d}.
\]
\end{thm}

\begin{thm}[Adaptation of Theorem \ref{thm: Useful for mixing times}]\label{thm: A_n Useful for mixing times}
For each $\epsilon>0$ there exist $c,C>0$ such that the following holds. Let \[\alpha \in\left(\frac{1}{c\log n}, 1-\frac{C\log \log n}{\log n}\right).\] Suppose that a class function $f\colon A_n\to \mathbb{R}$ satisfies $\|f\|_2 \le  n^{c\alpha n^\alpha}.$ Then for every $d>0$ and every character $\chi$ of level $d$ we have 
\[\frac{\left|\langle f,\chi\rangle \right|}{\chi(1)} \le 2\max\left(\epsilon^d, 2^{-n^{3/5}}\right) \chi(1)^{\alpha -1}.\] 
\end{thm}

\section{Mixing times}
\label{sec:mixing times}

In this section we prove our mixing result, namely Theorem \ref{thm:mixing time results}. Theorem \ref{thm:Simplified mixing time} will then follow immediately as it consists of the special case where all the functions are equal. We will then quickly deduce Corollary \ref{cor:two steps mixing}.
\subsection{Proof of Theorem \ref{thm:mixing time results}}

The standard way to deduce mixing time results from character bounds of the form $\frac{|\chi(\sigma)|}{\chi(1)}$ involves estimates of the Witten zeta function $\sum_{\chi}\chi(1)^{-s}$ and is well known. This route works smoothly when one is interested in rough estimates of the mixing time, but our more precise estimates require a different approach. 

\begin{proof}[Proof of Theorem \ref{thm:mixing time results}]
We may assume that $n$ is sufficiently large by increasing $C$ if necessary.
It is well known that we have \[\langle f_1*f_2*\cdots * f_\ell,\chi\rangle = \frac{\hat{f_1}(\chi)\cdots \hat{f_\ell}(\chi)}{\chi(1)^{\ell-1}}.\]
Since the characters form an orthonormal basis, we obtain
\[
    \|f_1*f_2*\cdots * f_\ell - 1\|_2^2 = \sum_{\chi \ne 1 }\frac{|\hat{f_1}(\chi)|^2\cdots |\hat{f_\ell}(\chi)|^2}{\chi(1)^{2\ell-2}}.
\]

Recall that $0 \le \alpha_i \le 1-\frac{C\log\log n}{\log n}$ for all $1 \le i \le \ell$ and $\sum_i\alpha_i \le \ell-1$. Let $c>0$ be a sufficiently small constant to be chosen later. Since the statement of the theorem gets weaker as $\alpha_i$ increases in the interval $(0,1-\frac{2C\log\log n}{\log n}]$, we may increase the $\alpha_i$ in the interval until either $\alpha_i \ge 1-\frac{2C\log\log n}{\log n}$ for all $i$ or $\sum_i\alpha_i = \ell-1$. In either case, we have $\alpha_i\ge \min((\ell-1)\frac{C\log \log n}{\log n}, 1-\frac{2C\log\log n}{\log n})\ge \frac{1}{c\log n}$, provided that $C,n$ are sufficiently large. By Theorem \ref{thm: A_n Useful for mixing times} with $\epsilon/16$ we have 
\[\|f_1*f_2*\cdots * f_\ell - 1\|_2^2 \le \sum_{d=1}^{n}\sum_{\chi \text{ of level }d}  4\max((\epsilon/16)^{d}, 2^{-n^{3/5}})^2 <  \epsilon,\] 
provided that $n$ is sufficiently large and $c$ is sufficiently small. Indeed, the number of characters of level $d$ is the same as the number of partitions of $d$ which is easily upper bounded by $2^d$, and there are at most $e^{O(\sqrt{n})}$ characters overall. 
\end{proof}


\subsection{Deducing Corollary \ref{cor:two steps mixing}}

\begin{proof}[Proof of Corollary \ref{cor:two steps mixing}] We have \[f * f (1) =\mathbb{E}_{\sigma \sim A_n}[f(\sigma)f(\sigma^{-1})] = \|f\|_2,\]
where the last equality follows from the fact that $f$ is real valued and symmetric. The corollary now follows from Theorem \ref{thm:mixing time results}.    
\end{proof}

\section{Product mixing}
\label{sec:product mixing}
In this section we show that $A_n$ is normally an $n^{-cn^{1/3}}$-mixer when $c$ is sufficiently small, but not normally an $n^{-Cn^{1/3}}$-mixer when $C$ is sufficiently large.

\begin{proof}[Proof of Theorem \ref{thm:normally mixing}]
    Let $A,B,C$ be subsets of $A_n$ of density $\ge n^{-cn^{1/3}}$. Write $f=\frac{1_A}{\mu(A)},g=\frac{1_B}{\mu(B)}$ and $h=\frac{1_C}{\mu(C)}$. Then we have \[\frac{\Pr_{a\sim A,b\sim B}[ab\in C] -\mu(C)}{\mu(C)} = \langle f*g,h\rangle - 1= \sum_{\chi\ne 1} \frac{\hat{f}(\chi)\hat{g}(\chi)\hat{h}(\chi)}{\chi(1)}.\] Applying Theorem \ref{thm: A_n Useful for mixing times} with $\alpha = 1/3$ completes the proof in similar fashion to the proof of Theorem \ref{thm:mixing time results}.
\end{proof}

\begin{proof}[Proof of Corollary \ref{cor:Nikolov pyber}]
Let $\sigma \in A_n$, and set $B=A$ and $C = A^{-1}\sigma^{A_n}.$ Since both $A,B$ and $C$ have density $\ge n^{-cn^{1/3}}$, we may apply Theorem \ref{thm:normally mixing} and obtain that $A B \cap C \ne \emptyset$. Therefore, there exists $a \in A$ such that $a^{-1} \sigma \in A^2$, which implies that $\sigma \in A^3$.
\end{proof}

\subsection{Sharpness example}

Our sharpness example consists of the set of permutations with about $n^{1/3}$ fixed points and a single $(n-n^{1/3})$-cycle.

\begin{proposition}\label{Prop:not a mixer}
    There exists an absolute constant $C>0,$ such that $A_n$ is not normally an $n^{-Cn^{1/3}}$-mixer.
\end{proposition}
\begin{proof} We may assume that $n$ is sufficiently large by increasing $C$ if necessary. Set  $A$ to be the set of permutations with $t = \lceil 10n^{1/3}+1\rceil$ fixed points and a single $(n-t)$-cycle. We will prove that 
\[
\frac{\Pr_{a,b\sim A}[ab\in A] -\mu(A)}{\mu(A)} >1.01.
\]

Let $f=g=h=\frac{1_A}{\mu(A)}$. We have \[\langle f*g,h\rangle -1 = \sum_{\chi\ne 1}\frac{\widehat{f}(\chi)^3}{\chi(1)}.\]

Consider the partition $(n-1,1)$. Since $\chi_{n-1,1}(\sigma)$ counts the fixed points of $\sigma$ minus 1, the term with the character $\chi_{n-1,1}$ contributes at least 1000 to the sum. Therefore, in order for the sum to be $o(1)$, the other terms must cancel it out. However, by the Murnaghan--Nakayama rule all the terms corresponding to characters of level $\le n^{1/3}$ are positive. Indeed, when removing a rim hook of length $n-t$ from a Young diagram of level at most $n^{1/3}$, the rim hook must be contained in the first row. We now upper bound the sum of the absolute values of terms of level $>n^{1/3}.$ For levels  $d\ge n^{5/12}$ we may combine the following facts to deduce that the sum over all characters of level $>n^{5/12}$ of $\frac{|\hat{f}(\chi)^3|}{\chi(1)}$ is $\le 1.$  Firstly, we have $|\hat{f}(\chi)|\le \|f\|_2\le n^{Cn^{1/3}}$ for some absolute constant $C>0$. Secondly, the number of irreducible characters of level $d$ is the number of partitions of $d$, which is $e^{O(\sqrt{d})}$ and overall the number of characters is $e^{O(\sqrt{n})}$. Finally, we have $\chi(1)\ge  \left(\frac{n}{ed}\right)^d$ when $d \le n/200$ and at least $e^{\Omega(n)}$ for $d>n/200$ by Theorem \ref{thm:EFP dimensions of representations}. 
To upper bound $|\hat{f}(\chi)|$ for the remaining $\chi$, we fix a $t$-umvirate $U_{I}$. Identifying $U_I$ with $S_{n-|I|}$ we let $\sigma$ be an $(n-|I|)$-cycle in $U_I$. We then have $\hat{f}(\chi) = \chi_{I\to I}(\sigma).$ We may now use Theorem \ref{prop:characters are global} to upper bound \[\frac{1}{n}\chi_{I\to I}(\sigma)^2 \le \|\chi_{I\to I}\|_2^2 \le 2^d \tilde{\chi}(1),\] where $\tilde{\chi}$ is obtained from $\chi$ by deleting its first row and using the fact that the density of the $(n-|I|)$-cycles is $\ge \frac{1}{n}$. We also have $\chi(1)\ge \binom{n-d}{d}\tilde{\chi}(1)$ by Lemma \ref{lem: dimension of characters}.
This shows that we have \[\frac{|\hat{f}(\chi)|^3}{\chi(1)} \le \frac{2^{3d/2}n^{3/2}\tilde{\chi}(1)^2}{\binom{n-d}{d}}.\]
We may now upper bound $\tilde{\chi}(1)\le \sqrt{d!}$ to obtain that
\[
\frac{|\hat{f}(\chi)^3|}{\chi(1)}\le 2^{-d}.
\]
Summing over all $\chi$ of levels in the interval $(n^{1/3}, n^{5/12})$ completes the proof as there are only $e^{O(\sqrt{d})}$ characters of level $d$ for each $d$.
\end{proof}

\bibliographystyle{plain}
\bibliography{refs}

\end{document}